\documentclass[10pt]{amsart}
\usepackage{}%
\usepackage{latexsym}
\usepackage{amscd}
\usepackage{amssymb}
\usepackage{graphicx}
\usepackage{epstopdf}
\usepackage{amsmath}
\usepackage{textcomp}

\usepackage{cancel}
\usepackage{tikz}
\usetikzlibrary{decorations.markings}
\tikzstyle{vertex}=[circle, draw, inner sep=0pt, minimum size=6pt]
\newcommand{\vertex}{\node[vertex]}

\usepackage{caption}
\usepackage{subcaption}
\usepackage{amssymb, amscd, amsmath,amsthm,amsfonts,epsfig, enumerate}
\usepackage{epsfig}
\usepackage{fullpage}
\usepackage{listings}
\usepackage{float}
\usepackage{epstopdf}
\usepackage{color}

\def\NZQ{\Bbb}               % the font for N,Z,Q,R,C

\def\ZZ{{\NZQ Z}}

\def\B'c{{\mathcal{B'}}}
\def\U'c{{\mathcal{U'}}}
\def\opn#1#2{\def#1{\operatorname{#2}}} % to make operators
\opn\chara{char}
\opn\length{\ell}
\opn\projdim{proj\,dim}
\opn\injdim{inj\,dim}
\opn\ini{in}
\opn\rank{rank}
\opn\depth{depth}

\opn\sdepth{sdepth}
\opn\indmat{indmat}
\opn\cochord{cochord}
\opn\pdim{pdim}
\opn\height{ht}
\opn\embdim{emb\,dim}
\opn\codim{codim}

\opn\Tr{Tr}
\opn\bigrank{big\,rank}
\opn\superheight{superheight}\opn\lcm{lcm}
\opn\trdeg{tr\,deg}%
\opn\reg{reg}
\opn\lreg{lreg}
\opn\set{set}
\opn\supp{Supp}
\opn\shad{Shad}
\opn\div{div}
\opn\Div{Div}
\opn\cl{cl}
\opn\Cl{Cl}
\opn\Spec{Spec}
\opn\Supp{Supp}
\opn\supp{supp}
\opn\Sing{Sing}
\opn\Ass{Ass}
\opn\Min{Min}
\opn\size{size}
\opn\bigsize{bigsize}
\opn\lex{lex}
\opn\Ann{Ann}
\opn\Rad{Rad}
\opn\Soc{Soc}
\opn\Ker{Ker}
\opn\Coker{Coker}
\opn\Im{Im}
\opn\Hom{Hom}
\opn\Tor{Tor}
\opn\Ext{Ext}
\opn\End{End}
\opn\Aut{Aut}
\opn\id{id}

\opn\nat{nat}
\opn\GL{GL}
\opn\SL{SL}
\opn\mod{mod}
\opn\ord{ord}
\opn\aff{aff}
\opn\con{conv}
\opn\relint{relint}
\opn\st{st}
\opn\lk{lk}
\opn\cn{cn}
\opn\core{core}
\opn\vol{vol}
\opn\gr{gr}

\def\pot#1#2{#1[\kern-0.28ex[#2]\kern-0.28ex]}

\opn\dirlim{\underrightarrow{\lim}}
\opn\invlim{\underleftarrow{\lim}}
\let\tensor=\otimes

\def\pnt{{\raise0.5mm\hbox{\large\bf.}}}

\def\Implies{\ifmmode\Longrightarrow \else
	\unskip${}\Longrightarrow{}$\ignorespaces\fi}
\def\implies{\ifmmode\Rightarrow \else
	\unskip${}\Rightarrow{}$\ignorespaces\fi}
\def\iff{\ifmmode\Longleftrightarrow \else
	\unskip${}\Longleftrightarrow{}$\ignorespaces\fi}

\let\:=\colon
\newtheorem{Theorem}{Theorem}[section]
\newtheorem{Lemma}[Theorem]{Lemma}
\newtheorem{Corollary}[Theorem]{Corollary}

\newtheorem{Remark}[Theorem]{Remark}

\newtheorem{Definition}[Theorem]{Definition}

\let\epsilon=\varepsilon
\let\phi=\varphi
\let\kappa=\varkappa
\numberwithin{equation}{section}
\title{Algebraic Properties of Edge Ideals of Corona Product of Certain Graphs}

\author[Bakhtawar Shaukat$^1$]{Bakhtawar Shaukat$^1$}
\author[Muhammad Ishaq$^1,*$]{Muhammad Ishaq$^{1,*}$}
\author[Ahtsham ul Haq$^1$]{Ahtsham ul Haq$^1$}
\author[Zahid Iqbal$^2$]{Zahid Iqbal$^2$}
\begin{document}
	\maketitle
 \begin{center}
     $^1$School of Natural Sciences, National University of Sciences and Technology Islamabad, Sector H-12, Islamabad Pakistan.\linebreak
     $^2$Department of Mathematics and Statistics, Institute of Southern Punjab, Multan, Pakistan.\linebreak
     $^*$ Corresponding email: \email{ishaq\_maths@yahoo.com}; Tel: +92-51-90855591\\
     Contributing authors emails: bakhtawar.shaukat@sns.nust.edu.pk, ahtsham2192@gmail.com, 786zahidwarraich@gmail.com
     \end{center}
 
  \begin{abstract}
		We study the algebraic invariants namely depth, Stanley depth, regularity and projective dimension of the residue class rings of the edge ideals associated with the corona product of various classes of graphs with any graph. We also give an expression for the Krull dimension of the residue class rings of the edge ideals associated with the corona product of any two graphs. In the end, we discuss when these graphs will be  Cohen-Macaulay.
		%\textbf{2010 Mathematics Subject Classification:} Primary: 13C15, Secondary: 13P10, 13F20.
	\end{abstract}
	\section*{Introduction}
	Let $K$ be a field and  $S:=K[x_1, \dots , x_r]$ denote the polynomial ring in $r$ variables over $K$ with standard grading. For a finitely generated $\mathbb{Z}$-graded $S$-module $A$ with a minimal free resolution 
	$$0\longrightarrow\ \bigoplus_{j\in \ZZ} S(-j)^{\beta_{p,j}(A)} \longrightarrow \bigoplus_{j\in \ZZ}  S(-j)^{\beta_{p-1,j}(A)}
	\longrightarrow \dots 	\longrightarrow\bigoplus_{j\in \ZZ}  S(-j)^{\beta_{0,j}(A)}\longrightarrow A \longrightarrow\ 0,$$
the \textit{regularity}
and \textit{projective dimension} of $A$ are defined  by   
$\reg(A)=\max\{j-i:\beta_{i,j}(A)\neq 0\}$ 
and	  $\pdim(A)=\max\{i:\beta_{i,j}(A)\neq 0\},$  respectively.
%	\begin{equation*}	\pdim(A)=\max\{i:\beta_{i,j}(A)\neq 0\}.\end {equation*}
	%	We say	$A$ has a \textit{$d$-linear resolution} over $S$ if associated to $A$ is a minimal graded free resolution of the form	$$0\longrightarrow S(-d-q)^{\beta_{q}} \longrightarrow \dots \longrightarrow S(-d-1)^{\beta_{1}} \longrightarrow  S(-d)^{\beta_{0}}\longrightarrow A \longrightarrow 0,$$ where $q\leq r.$ In other words, if the graded Betti numbers of $A$ have the property $\beta_{i,j}(A)$=0 for all $ i,j$ such that $j-i \neq d.$ 
Regularity plays a significant role as one of the keys indicators of a module's complexity and is an
important invariant in commutative algebra. 
Values and bounds for the regularity and projective dimension of edge ideals have been studied by a number of researchers; see for instance  \cite{regnpro,regul,faridi,vantyl,hibb}. 	
		%		Also the interplay of algebraic properties of edge ideal and graph invariants is of great interest. Such an interplay can be seen in \cite{ susanmoreytree,moreyvir,book,woodroof}.

  Let $\mathfrak{m}:=(x_1,\dots,x_r)$ be the unique graded maximal ideal of $S$. The depth of $A$ is defined to be the common length of all maximal $M$-sequences contained in  $\mathfrak{m}$, for more details related to depth we refer the readers to \cite{depth}. For   a finitely generated $\mathbb{Z}^{n}$-graded $S$-module $A,$ if $a_{i}\in A$ is a homogeneous element  and  $X_{i}\subset \{x_{1}, \ldots, x_{r}\},$ then  
	$a_{i}K[X_{i}]$ denotes the $K$-subspace of $A$ generated
	by all homogeneous elements of the form $a_it$, where $t$ is a monomial in $K[X_{i}]$. The linear $K$-subspace $a_{i}K[X_{i}]\subset A$ is
	called a Stanley space of dimension $|X_{i}|$ if $a_{i}K[X_{i}]$ is
	a free $K[X_{i}]$-module, where $|X_{i}|$ denotes the number of indeterminates in $X_{i}$.  A Stanley decomposition of $A$ is a	presentation of the $K$-vector space $A$ as a finite direct sum of Stanley spaces
	   $\mathcal{D}: A = \bigoplus_{i = 1}^{r}a_{i}K[X_{i}].$ The Stanley
	depth of decomposition $\mathcal{D}$ and Stanley depth	of $A$ are defined as
	$\sdepth(\mathcal{D}) = \min\{|X_{i}|: i = 1, \ldots, r\}$ and
	$\sdepth(A) = \max\{\sdepth(\mathcal{D}): \text{$\mathcal{D}$~is~a~Stanley~decomposition~of
		~$A$}\},$ respectively. Herzog et al. showed that the Stanley depth of $A = J/I,$
	where $I\subset J\subset S$  can be computed in a finite number of steps by using some posets related to $A$ in \cite{HVZ}. However, it is a difficult task to compute the Stanley depth even by using their
	method. %In the same paper they tried to find $\sdepth(\mathfrak{m})$ by using their method but  they were  unable to find the value in general. However, on the basis of their observation for small values of $r$ they conjecture that $\sdepth(\mathfrak{m})=\lceil\frac{r}{2}\rceil,$ where $\lceil t\rceil,$ $t\in Q$, means the smallest integer which is not less than $t$.   This conjecture was proved by  Biro et al.  in \cite{BH}. 
	For some known results, related to values and bounds of depth and Stanley depth, we refer the readers to \cite{MC8,IQ,ZIA,PFY,bakht}. Stanley in \cite{RP} conjectured  that $\sdepth(A)\geq \depth(A).$  %This conjecture  was interesting in the sense that it compares two invariants of different nature.
	This conjecture was proved for some special cases; see for instance \cite{fakh,fouliii,susanmoreytree,AR1}. Later on, this conjecture was disproved by Duval et al. in \cite{DG} by providing a counterexample of $S/I$, where $I\subset S$ is a monomial ideal.

%The study of  depth, Stanley depth, regularity, projective dimension and Krull dimension of modules  represent the current trend in active area of commutative algebra, see \cite{regul,hibii,susanmoreytree,circulent,woodroof}. 

		For a simple graph $G=(V(G), E(G))$ with vertex set $V(G)=\{x_1, \dots , x_r\}$ and the edge set $E(G),$ we identify the vertices of a graph and variables in $S$. The \textit{edge ideal} of a graph $G$ is defined as $I(G)=(x_{i}x_{j}:\{x_i,x_j\}\in E(G)) \subset S$.  The module $S/I(G)$ is called  Cohen–Macaulay if $\depth(S/I(G))=\dim(S/I(G))$. A graph $G$ is said to be Cohen–Macaulay if $S/I(G)$ is Cohen–Macaulay. Characterization and construction of  Cohen–Macaulay graphs is one of the fundamental  problems with rich literature; see for instance \cite{ asms, asms2,vir2,tehran,imrananwar, circcohen}. 
A broad categorization of Cohen–Macaulay graphs is difficult.	 Villarreal characterized all Cohen–Macaulay trees  \cite{vilareal}. All Cohen–Macaulay chordal graphs are classified by Herzog et al. in \cite{herhib} and later on Herzog and Hibi classified all Cohen–Macaulay bipartite graphs   \cite{hh}. The primary goal of this paper is to compute some algebraic invariants and as a result of our findings, we characterize some Cohen–Macaulay graphs.
		%		This paper is an extension of the work in [3] where the focus is on the second power of the ideal I.  
		%		The study of algebraic property of edge ideals  (e.g., Cohen–Macaulayness) is of current interest. In this paper we will focus on the following algebraic properties: the sequentially Cohen–Macaulay property, the stability of associated primes, and the connection between torsion-freeness and combinatorial problems. The numerical invariants of edge ideals have attracted a great deal of interest [1, 46, 53,77,87,90,91,106,110,117,118]. In this paper we focus on the following invariants: projective dimension, regularity, depth and Krull dimension. We present a few new results on edge ideals. We give a criterion to estimate the regularity of edge ideals (see Theorem 3.14). We apply this criterion to give new proofs of some formulas for the regularity of edge ideals (see Corollary 3.15).

		%	In this paper we focus on the following invariants: depth, Stanley depth, regularity, projective dimension,  and Krull dimension. The study of algebraic property of edge ideals  (Cohen–Macaulayness) is of current interest. This paper is an extension of Villarreal's work (whiskered graph are Cohen-Macaulay) in  \cite{vilareal}   generalizes the notion of whiskering given in Theorem \ref{cohnn} to identify  new Cohen-Macaulay graphs. Corona product of star graph with any graph  also generalizes the results on few mentioned invaiants computed in \cite{naem}.
	
In this paper, we consider the edge ideals associated to the graph  $X\odot H$ called the corona product of $X$ and $H.$   Here $H$  could be any graph and $X$ is either a path, a cycle, a complete graph, a star graph, or a complete bipartite graph. 
Particularly, if $H$ is a null graph, then the corona product $X\odot H$ is a  $|V(H)|$-fold bristled graph. The first section of this paper  contains some background  regarding Graph Theory, relevant useful results from Commutative Algebra   and all the concepts that will be used throughout the paper. In the second section,  we compute depth, projective dimension and lower bound of the Stanley depth for  the residue class rings of the edge ideals of $X\odot H$ in terms of corresponding algebraic invariants for the residue class ring of the edge ideal of $H;$ see for instance Theorem \ref{The1}, Theorem \ref{The3}, Theorem \ref{The119}, Theorem \ref{The1199} and Theorem \ref{The1177}.  In particular, if $H$ is a null graph, then the exact values for Stanley depth are determined. Moreover, we give an explicit formulas for the regularity of residue class rings of the edge ideals of $X\odot H,$ see Theorem \ref{THE2}, Theorem \ref{regcn}, Theorem \ref{regkn}, Theorem \ref{starreg} and Theorem \ref{combi}. Third section is devoted to find an expression for Krull dimension of edge ideal of corona product of any two graphs see Theorem \ref{krulldim}. As a consequence of  this theorem, we  characterize  Cohen-Macaulay $X\odot H$  graphs; see for instance Theorem \ref{cohnn}.

The corona product $X\odot H$ gives rise to a variety of nice graph structures. %Indeed, to do computation by considering $X$ to be an arbitrary class is a challenging task.  In future, if one consider $X$ to be some general class and define more complex structures and compute the algebraic invariants, then our work will serve as a basic cases for them.
 Some authors have already discussed depth and Stanley depth of the cyclic modules associated to some specific classes of corona products of graphs; see for instance \cite{naem}. To find an expression for the algebraic invariants of the  residue class rings of the edge ideals of $X\odot H$ seems to be a challenging task, if $X$ is any arbitrary graph. 
%Computing the aforementioned invariants can be a challenging task, particularly for the  residue class rings of the edge ideals of complex graph structures $X\odot H,$ if $X$ is any graph.
Therefore, we initially focus on fixing the underlying graph $X$ to be some well-known graphs, that allow us to gain insight into the properties of these structures and their associated algebraic invariants.
Our work may help to extend all
the results obtained in this paper to compute the said invariants when $X$ is arbitrary. To this end, we have succeeded in finding an expression for the value of the Krull dimension of the residue class rings of the edge ideal of the graphs $X\odot H$.
		\section{Preliminaries}\label{section1}
		In this section, we recall some notions from Graph Theory, review a few auxiliary findings from Commutative Algebra and  introduce some notations that are needed in the next sections.
		%	Let $K$ be a field and $S:=K[x_1,\dots,x_n]$ the polynomial ring over $K$. We begin by reviewing some standard notation and terminology regarding graphs	and their connections to algebra. By abuse of notation, $x_{i}$ will be used to denote	both the vertex of a graph $G$ and the corresponding variable of the polynomial ring $S.$ Throughout the paper all graphs will be assumed to be simple. Let $G:=(V(G),E(G))$ be a graph with vertex set $V(G)=\{x_1,x_2,\dots,x_r\}$ and edge set $E(G)$. Then the edge ideal $I(G)$ associated to $G$ is the square free monomial ideal of $S$, that is $I(G)=(x_{i}x_{j} \,\,\,|\,\,\, \{x_i,x_j\}\in E(G)).$
		%	Let $G=(V(G),E(G))$ be a graph with vertex set $V(G)$ and edge set $E(G)$. A graph is said to be \textit{simple} if it has no loops and multiple edges.  
		Throughout this paper, all graphs are assumed to be finite, simple and may have isolated vertices. 	For a graph $G$ with $V(G)=\{x_1, \dots, x_n\},$	adding a \textit{whisker} to graph $G$ at $x_i$ means adding a new vertex $y$ to $V(G)$ and edge $\{x_i,y\}$ to $E(G).$ A graph obtained by adding a whisker to each vertex of any given graph is called a \textit{whisker graph}.	A graph $N_{r}$ is said to be a  \textit{null graph} on $r$ vertices if $V(N_{r})=\{x_{1},\dots , x_{r}\}$ and $E(N_{r})=\emptyset$. Moreover, if $r=1$ then $N_1$ is also called a \textit{trivial graph}. The \textit{degree} of a vertex $x_i\in V(G)$ is the number of adjacent vertices to $x_i$ in graph $G$ and  is represented by $\deg_G(x_i)$. Any vertex of degree $1$ is called a \textit{leaf} or \textit{pendant vertex} of $G.$  For any integer $t\geq 1,$ a $t$-\textit{fold bristled graph} of $G$ is formed by attaching $t$ pendant vertices to each vertex of $G$. An  \textit{internal vertex} is a vertex that is not a leaf. A graph  is said to be a \textit{connected graph} if there is a path between any two vertices. For $r\geq 1,$ a \textit{path} $P_{r}$ is a graph on $r$ vertices  %, say $x_{1},\dots,x_{r},$ 
		such that $E(P_{r})=\{\{x_{j},x_{j+1}\}:1\leq j\leq r-1\}$  (if $r=1,$ then $E(P_1)=\emptyset$). For $r\geq 3,$ a \textit{cycle} $C_{r}$ on $r$ vertices is a graph such that $E(C_{r})=\{\{x_{j},x_{j+1}\}:1\leq j\leq r-1\}\cup \{x_{1},x_{r}\}.$ Let $r\geq 1,$ a \textit{complete graph}  $K_{r}$ on  $r$ vertices  is a graph in which each pair of vertices is connected by an edge (if $r=1,$ then $E(K_1)=\emptyset$).	A simple and connected graph $T_{r}$ on $r$ vertices is said to be a \textit{tree}  if there exists a unique path between any two vertices of $T_{r}$. For  $k\geq 1,$  a  \textit{$k$-star} is a tree with $k$ leaves and a single vertex with degree $k.$ We denote a $k$-star by $S_{k}.$ %and if $V(S_{k})=\{x_{1},\dots, x_{k+1}\}$, then  $E(S_{k})=\{\{x_{1},x_{j}\} :2\leq j\leq k+1\}$.  % The minimal set of monomial generators of monomial ideal $I(S_{n+1})$ is $\mathcal{G}(I(S_{n+1}))=\{y_{1}y_{2},y_{1}y_{3},\dots,y_{1}y_{n+1}\}.$	
		A \textit{bipartite graph} is a graph in which the set of vertices is partitioned into two disjoint sets called partite sets such that no two vertices of the graph within the same partite set are adjacent.
		A \textit{complete bipartite graph} is a bipartite graph such that every vertex of one partite set is connected to  each vertex of the other partite set.	%	every pair of graph vertices in the two sets are adjacent. 
	Let $K_{u,v}$ denotes the complete bipartite graph with partite sets $K_u=\{x_{1},\dots, x_{u}\}$ and $K_v=\{x_{u+1},\dots, x_{u+v}\}.$ A vertex $x_j$ is a \textit{neighbor} of a vertex $x_i$ in a graph $G$ if $\{x_i,x_j\}\in E(G).$ The \textit{neighborhood} $N_G(x_i)$ of a vertex $x_i$ is the set of all neighbors of $x_i,$ that is, 	$N_G(x_i):=\{x_j\in V(G)\,\,\,|\,\,\,\{ x_i,x_j\}\in E(G)\}.$ 
		For  graph $G$, a subset $W$ of $V(G)$ is called an \textit{independent set} if no two vertices in $W$ are adjacent. The cardinality of the largest independent
		set of $G$ is called the $independence \ number$ of $G$.	\begin{Lemma}[{\cite[Lemma 1]{dim}}]\label{prok}
			$\dim (S/I(G))=\max\{|W|:\text{ $W$ \text{is an independent set of} $G$}\}.$\end{Lemma} % \noindent For any module $A$,	an elementary property of depth is the following inequality:\begin{equation}\label{ddim}	\depth(A)\leq \dim(A).	\end{equation}
			\noindent A \textit{subgraph} $H$ of a graph $G$, written as $H \subseteq G$, is a graph such that $V(H) \subseteq V(G)$ and  $E(H) \subseteq E(G)$.
		For a subset $ U\subseteq V(G)$, an \textit{induced subgraph} of $G$ is a graph $G':=(U, E( G'))$, such that $E( G')= \{\{x_i,x_j\}\in E(G): \{x_i,x_j\}\subseteq U\} $. A graph is said to be  \textit{chordal} if it does not contain any induced cycle of length strictly greater than $3$.
	%	Let $G^{c}$ be the complement of the graph $G$ that is the graph with the	same vertex set as $G$ and $\{x_{i}, x_{j}\}$ is an edge of $G^{c}$	if it is not an edge of $G$. 
		A \textit{matching} $M$ in a graph $G$ is a subset  of $E(G)$ in which no two edges are adjacent in $G$. An \textit{induced matching} in $G$ is a matching that forms an induced subgraph of $G$. An $induced \ matching\ number$ of $G$ is denoted as $\indmat(G)$ and  defined as
		  $$\indmat(G)=\max \{|M|: \, \, \text{$M$ is an induced matching in $G$}\}.$$ 
		%	If $G$ is a finite simple graph, Katzman proved in \cite[Lemma 2.2]{kat} that $\indmat(G)$ is a lower bound for the regularity of $S/I(G)$ and then Hà et al. proved in \cite[Corollary  6.9]{hanvan} that regularity of $S/I(G)$ is equal to the $\indmat(G)$ if $G$ is a chordal graph. We combine these results in the following lemma.	[{\cite[Theorems 2.6 and 2.7]{AA}}]
		\begin{Lemma}[{\cite[Corollary  6.9]{hanvan}}]\label{reg1}
		 If $G$ is a \text{chordal graph, then} $\reg(S/I(G))=\indmat(G) .$ 
		\end{Lemma}
	%	\begin{Proposition} 	Let $G$ be a finite simple graph. Then \begin{itemize}	\item[(a)] $\reg(S/I(G))\geq \indmat(G),$ {\cite[\text{Lemma }2.2]{kat}}.	\item[(b)] If $G$ is a \text{chordal graph, then} $\reg(S/I(G))=\indmat(G) ,$ {\cite[\text{Corollary }6.9]{hanvan}}. 
				%\item[c] Over any field $k$, $\reg(S/I(G))\leq \cochord(G)$	\end{itemize}	\end{Proposition}
		%	It is well known that one may compute the Krull dimension of $S/I(G)$ by using independent sets.  

	%	\begin{Lemma}[{\cite{5}}]\label{fro}	If $G$ is a graph, then $I(G)$ has a linear	resolution if and only if $G^{c}$ is a chordal graph.	\end{Lemma}
		%A.Rauf proved the following lemma for Stanley depth.
		%[{\cite[Lemma 2.2]{AR1}}]
	%	\noindent Note that for any homogeneous ideal $I$ of $S$, $\reg (S/I)=\reg (I)-1$ and $\pdim (S/I)=\pdim (I)+1$.
\noindent	Next we present a few results,  these results will play a key role in the proofs of our main results.
	\begin{Lemma}[{\cite[Theorem 2.6  and Theorem 2.7]{AA}}]\label{leAli}
			If $r\geq 2$ and $I(S_{r-1})\subset S=K[V(S_{r})]$, then $$\depth(S/I(S_{r-1}))=\sdepth(S/I(S_{r-1}))=1.$$
		\end{Lemma}

%\noindent The next lemma combines the two findings for depth and Stanley depth from {\cite[Proposition 1.2.9]{depth}} and {\cite[Lemma 2.2]{AR1}}, respectively. 
		\begin{Lemma}\label{rle2}	
			If $0\longrightarrow X\longrightarrow Y \longrightarrow Z \longrightarrow 0$ is an exact sequence of $\mathbb{Z}^n$-graded $S$-module, then
			\begin{itemize}
			    \item[(a)]  $\depth (Y) \geq \min\{\depth(X), \depth(Z)\},$ {\cite[Proposition 1.2.9]{depth}}.
			     \item[(b)]  $\sdepth(Y)\geq\min\{\sdepth(X), \sdepth(Z)\},$  {\cite[Lemma 2.2]{AR1}}.
			\end{itemize}
		\end{Lemma}
		
%\noindent The following result was proved for depth by Rauf {\cite[Corollary 1.3]{AR1}},  while  a similar result for Stanley depth was proved by Cimpoeas {\cite[Proposition 2.7]{MC}}.
		\begin{Lemma}\label{Cor7}
			Let $I\subset S$ be a monomial ideal and $f$ be a monomial in $S$ such that $f\notin I.$ Then
			\begin{itemize}
			    \item[(a)] $\depth (S/(I : f))\geq \depth(S/I),$ {\cite[Corollary 1.3]{AR1}}.
			    
			    \item[(b)] $\sdepth (S/(I : f))\geq \sdepth(S/I),$ {\cite[Proposition 2.7]{MC}}.
			\end{itemize}
		\end{Lemma} 
%	\noindent The conclusion of {\cite[Corollary 3.3(i)]{regul}} is the next result and it also holds for an arbitrary monomial $f.$  	This result is also proved for squarefree monomial ideals in {\cite[Lemma 5.1]{dao}}. Their proof is based on a result of Kummini {\cite{kummini}} on multigraded Betti numbers of squarefree monomial ideals.
		\begin{Lemma}[{\cite[Theorem 4.3]{regul}}]\label{exacther}
		    Let $I$ be a monomial ideal and let $f$ be an arbitrary monomial in $S.$ Then 
		    $$\depth(S/I)=\depth(S/(I:f))  \,\,\text{if}\,\, \depth(S/(I,f))\geq \depth(S/(I:f)).$$
		\end{Lemma} \noindent We have a similar result for Stanley depth in the next lemma.
			\begin{Lemma}\label{exacthersdepth}
		    Let $I$ be a monomial ideal and let $f$ be a monomial in $S$ such that $f\notin I.$ Then 
		    $$\sdepth(S/I)=\sdepth(S/(I:f))  \,\,\text{if}\,\, \sdepth(S/(I,f))\geq \sdepth(S/(I:f)).$$
		\end{Lemma}
		\begin{proof} 
		   Consider the short exact sequence $$0\longrightarrow S/(I:f)\xrightarrow{\,\,.f\,} S/I \longrightarrow S/(I,f) \longrightarrow 0.$$  By Lemma \ref{rle2}, $\sdepth (S/I) \geq \min\{\sdepth(S/(I:f)), \sdepth(S/(I,f))\}.$ If $\sdepth(S/(I,f))\geq \sdepth(S/(I:f)),$ then $\sdepth(S/I)\geq \sdepth(S/(I:f)).$  By Lemma \ref{Cor7}, we get $\sdepth(S/I)\leq \sdepth(S/(I:f)).$ Thus the required result follows.
		\end{proof}
	%	In the next lemma, we recall an inequalities for depth from Depth Lemma {\cite[Proposition 1.2.9]{depth}} and {\cite[Corollary 1.3]{AR1}}.  A similar results for Stanley depth proved by Rauf {\cite[Lemma 2.2]{AR1}} and in {\cite[Proposition 2.7]{MC}}. 
%\noindent An intriguing aspect of depth and Stanley depth is that they increase as new variables are added to the ring  {\cite[Lemma 3.6]{HVZ}}, but regularity stays the same {\cite[Lemma 3.5]{moreyvir}}. The following lemma provides these findings.
%	\noindent We know that $\depth(S)=\sdepth(S)=r$ and $\reg(S)=0$.
		\begin{Lemma}\label{le3}
			Let $I\subset S$ be a monomial ideal, and $\hat{S}=S \tensor_{K} K[x_{r+1}]$. Then 
			\begin{itemize}
			    \item[(a)] $\depth(\hat{S}/I)=\depth(S/I)+1$  and $\sdepth(\hat{S}/I)=\sdepth(S/I)+1,$  {\cite[Lemma 3.6]{HVZ}}.
			    \item[(b)] $\reg({\hat{S}/I})=\reg({S/I}),$ {\cite[Lemma 3.6]{moreyvir}}.
			\end{itemize}
		\end{Lemma} 	

\begin{Lemma}[{\cite[Proposition 2.2.20]{book}}]\label{virtens}
     Let $1\leq n < r.$   If $S_{1}=K[x_{1},\dots, x_{n}]$ and $S_{2}=K[x_{n+1},\dots, x_{r}].$ Then  $S/(I+J)\cong S_{1}/I \tensor_{K}S_{2}/J.$ 
\end{Lemma}

\noindent	 We use Lemma \ref{virtens} and combine it with {\cite[Proposition 2.2.21]{book}} and  {\cite[Theorem 3.1]{AR1}} for depth and Stanley depth, respectively and get the following useful result.
		\begin{Lemma}\label{LEMMA1.5}
			$ \depth_{S}(S_{1}/I\tensor_{K}S_{2}/J)= \depth_{S}(S/(I+J))=\depth_{S_{1}}(S_{1}/I)+\depth_{S_{2}}(S_{2}/J)$ and 	$ \sdepth_{S}(S_{1}/I \tensor_{K} S_{2}/J) \geq \sdepth_{S_{1}}(S_{1}/I)+\sdepth_{S_{2}}(S_{2}/J).$
		\end{Lemma}
	%	\noindent Additionally, Woodroofe proved the following lemma if $I$ and $J$ are two edge ideals minimally generated by disjoint sets of variables.
		
		\begin{Lemma}[{\cite[Theorems 1.3.3]{depth}}](Auslander–Buchsbaum formula)\label{auss13}
			If $ R $ is a commutative Noetherian local ring and M is a non-zero finitely generated R-module of finite projective dimension, then
			\begin{equation*}
				{\pdim}(M)+{\depth}(M)={\depth}(R).
			\end{equation*}
		\end{Lemma}
\begin{Lemma}[{\cite[Lemma 3.2]{HOA}}]\label{circulentt}
       If $I\subset S_{1}=K[x_{1},\dots, x_{n}]$ and $J\subset S_{2}=K[x_{n+1},\dots, x_{r}]$ are non-zero homogeneous ideals of  $S_1$ and $S_2$ and regard $I+J$ as a homogeneous ideal of $S.$ Then
$$\reg({S/I+J})=\reg(S_1/I)+\reg(S_2/J).$$
%Let $G$ be a finite simple graph and $S:=K[V(G)]$. If $G=H\cup K,$ with $H$ and $K$ disjoint, then $$\reg({S/I(G)})=\reg(K[V(H)]/I(H))+\reg(K[V(K)]/I(K)).$$
%Let $S_1=K[x_1,\dots,x_n]$ and $S_2=K[x_{n+1},\dots,x_r]$  and $I$ and $J$ be edge ideals of $S_1$ and $S_2$, respectively. Then $$\reg({S/(I+J)})=\reg(S_{1}/I)+\reg(S_{2}/J).$$
\end{Lemma}

		\noindent %We now refer back to a regularity related lemma in the form given in {\cite[Theorem 4.7]{regul}}. 
	In the next lemma, proof of	parts (a) and (c)  follows  from Corollary 20.19 and Proposition 20.20 of \cite{eisenbud}, %{\cite[Corollary 20.19 and Proposition 20.20]{eisenbud}},
	while  part (b) follows  from {\cite[	Lemma 2.10]{dao}}. 
		\begin{Lemma}[{\cite[Theorem 4.7]{regul}}]\label{regul2}%[{\cite[Theorems 4.7]{regul}}]
			Let $ I $ be a monomial ideal and $ x_{i} $ be a variable of $S$. Then
			
			\begin{itemize}
				\item [(a)]  $ \reg (S/I) = \reg (S/(I:x_{i}))+1$, if $ \reg (S/(I:x_{i})) > \reg (S/(I,x_{i})),$
				\item[(b)] 	  $\reg (S/I)\in \{\reg (S/(I,x_{i}))+1,\reg (S/(I,x_{i}))\},$ if $ \reg (S/(I:x_{i})) = \reg (S/(I,x_{i})),$
				\item[(c)] 	  $\reg (S/I)= \reg (S/(I,x_{i}))$ if $ \reg (S/(I:x_{i})) < \reg (S/(I,x_{i})).$
			\end{itemize}
		\end{Lemma}

			\begin{Definition}[{\cite{coronadef}}]\label{defcorona}
			\em{	The corona product of two graphs $X$ and $H$ denoted by $X\odot H$  is a graph constructed by taking one copy of graph $X$
				 and $|V(X)|$ copies of $H,$ namely, $H_{1}, H_{2}, \dots, H_{|V(X)|}$ and then connecting the $i^{th}$ vertex of graph $X$ to every vertex in $H_{i}.$} \end{Definition}\noindent If $X$ and $H$ are two graphs, then
		\begin{itemize}
			\item [$\bullet$] $|V(X\odot H)|=|V(X)|(|V(H)|+1)$,
			\item [$\bullet$]$|E(X\odot H)|=|E(X)|+|V(X)||E(H)|+|V(X)||V(H)|$.
				\item[$\bullet$] If $X\neq H$ then $X\odot H\neq H\odot X$.
		\end{itemize}
 	 For a monomial ideal $I$ we denote the minimal set of monomial generators of $I$ by $\mathcal{G}(I).$  	Let $|V(X)|=n$  and $|V(H)|=m.$  Throughout this paper, when we consider  $X\odot H$, the vertices of $X$ are labeled as   $y_1,y_2,\dots, y_n.$ Let $x_1,x_2,\dots, x_m$ be the vertices of $H.$  The vertices of the $i^{th}$ copy $H_i$ of  $H$ in $X\odot H$ are labeled as $x_{i1},x_{i2},\dots,x_{im}$ such that $\{x_{il},x_{ik}\}\in E(H_i)\, \text{iff} \, \{x_l,x_k\}\in E(H).$
%$x_{il}$ and $x_{ik}$ are adjacent in $H_i$ iff $x_l$ and $x_k$ are adjacent in $H.$ 
Thus 
	$V(X\odot H)=\{y_1,y_2,\dots, y_n\}\underset{i=1}{\overset{n}{\bigcup}} \{x_{i1},x_{i2},\dots,x_{im}\}$ 
 and  $$	E(X\odot H)=E(X)\underset{i=1}{\overset{n}{\bigcup}}\big\{\{y_i,x_{i1}\},\{y_i,x_{i2}\},\dots,\{y_i,x_{im}\}\big\}\underset{i=1}{\overset{n}{\bigcup}}E(H_i).$$ The minimal set of monomial generators of a monomial ideal $I(X\odot H)$ is as follow:	\begin{eqnarray*}
	    \mathcal{G}(I(X\odot H))=	\mathcal{G}(I(X))\underset{i=1}{\overset{n}{\bigcup}}\{y_ix_{i1},y_ix_{i2},\dots,y_ix_{im}\}\underset{i=1}{\overset{n}{\bigcup}}\mathcal{G}(I(H_i)).
	\end{eqnarray*}   For example, if $X=P_3$ and $H=T_7,$ where $T_7$ is a tree on $7$ vertices, then the graph $X\odot H$ is shown in Figure  \ref{coro}.
	
		\begin{figure}[H]
			\centering
			\begin{subfigure}[b]{0.45\textwidth}
				\centering
				\[\begin{tikzpicture}[x=1.1cm, y=0.9cm]
					\vertex[fill] (1) at (-0.5,-2) [label=below:$y_{1}$] {};
					\vertex[fill] (2) at (1,-2) [label=below:$y_{2}$] {};
					\vertex[fill] (3) at (2.5,-2) [label=below:$y_{3}$] {};
					\vertex[fill] (v1) at (0,0) [label=below:$x_{1}$] {};
					\vertex[fill] (v5) at (0,1) [label=left:$x_{2}$] {};
					\vertex[fill] (v6) at (1,1) [label=above:$x_{3}$] {};
					\vertex[fill] (v9) at (0.5,2) [label=above:$x_{4}$] {};
					\vertex[fill] (v10) at (2,2) [label=right:$x_{5}$] {};
					\vertex[fill] (v3) at (2,3) [label=right:$x_{6}$] {};
					\vertex[fill] (v4) at (1,3) [label=right:$x_{7}$] {};
					\path 
					(v1) edge (v5)
					(v5) edge (v6)
					(v5) edge (v9)
					(v6) edge (v10)
					(v3) edge (v10)
					(v4) edge (v10)

					(1) edge (2)
					(2) edge (3)
					;
				\end{tikzpicture}\]
				
				\caption{$P_{3}$ and $T_{7}$ }
			\end{subfigure}
			\hfill
			\begin{subfigure}[b]{0.5\textwidth}
				\centering
				\[\begin{tikzpicture}[x=0.73cm, y=0.7cm]
					\vertex[fill] (1) at (-3,-2) [label=below:$y_{1}$] {};
					\vertex[fill] (2) at (0,-2) [label=below:$y_{2}$] {};
					\vertex[fill] (3) at (3,-2) [label=below:$y_{3}$] {};
					\vertex[fill] (u1) at (-4,0) [label=left:$x_{11}$] {};
					\vertex[fill] (u5) at (-4,1) [label=left:$x_{12}$] {};
					\vertex[fill] (u6) at (-3,1) [label=above:$x_{13}$] {};
					\vertex[fill] (u9) at (-3.5,2) [label=above:$x_{14}$] {};
					\vertex[fill] (u10) at (-2,2) [label=right:$x_{15}$] {};
					\vertex[fill] (u3) at (-2,3) [label=right:$x_{16}$] {};
					\vertex[fill] (u4) at (-2.65,3) [label=above: $ x_{17}$] {};
					\vertex[fill] (w1) at (-1,0) [label=left:$x_{21}$] {};
					\vertex[fill] (w5) at (-1,1) [label=left:$x_{22}$] {};
					\vertex[fill] (w6) at (0,1) [label=above:$x_{23}$] {};
					\vertex[fill] (w9) at (-0.5,2) [label=above:$x_{24}$] {};
					\vertex[fill] (w10) at (1,2) [label=right:$x_{25}$] {};
					\vertex[fill] (w3) at (1,3) [label=right:$x_{26}$] {};
					\vertex[fill] (w4) at (0.35,3) [label=above:$x_{27}$] {};
					\vertex[fill] (v1) at (2,0) [label=left:$x_{31}$] {};
					\vertex[fill] (v5) at (2,1) [label=left:$x_{32}$] {};
					\vertex[fill] (v6) at (3,1) [label=above:$x_{33}$] {};
					\vertex[fill] (v9) at (2.5,2) [label=above:$x_{34}$] {};
					\vertex[fill] (v10) at (4,2) [label=right:$x_{35}$] {};
					\vertex[fill] (v3) at (4,3) [label=right:$x_{36}$] {};
					\vertex[fill] (v4) at (3.35,3) [label=above:$x_{37}$] {};
					\path 
					(u1) edge (u5)
					(u5) edge (u6)
					(u5) edge (u9)
					(u6) edge (u10)
					(u3) edge (u10)
					(u4) edge (u10)
					(w1) edge (w5)
					(w5) edge (w6)
					(w5) edge (w9)
					(w6) edge (w10)
					(w3) edge (w10)
					(w4) edge (w10)
					(v1) edge (v5)
					(v5) edge (v6)
					(v5) edge (v9)
					(v6) edge (v10)
					(v3) edge (v10)
					(v4) edge (v10)
					(u1) edge (1)
					(u5) edge (1)
					(u5) edge (1)
					(u6) edge (1)
					(u3) edge (1)
					(u4) edge (1)
					(w1) edge (2)
					(w5) edge (2)
					(w5) edge (2)
					(w6) edge (2)
					(w3) edge (2)
					(w4) edge (2)
					(v1) edge (3)
					(v5) edge (3)
					(v5) edge (3)
					(v6) edge (3)
					(v3) edge (3)
					(v4) edge (3)
					(1) edge (u5)
					(1) edge (u6)
					(1) edge (u9)
					(1) edge (u10)
					(1) edge (u10)
					(1) edge (u10)
					(2) edge (w5)
					(2) edge (w6)
					(2) edge (w9)
					(2) edge (w10)
					(2) edge (w10)
					(2) edge (w10)
					(3) edge (v5)
					(3) edge (v6)
					(3) edge (v9)
					(3) edge (v10)
					(3) edge (v10)
					(3) edge (v10)
					(1) edge (2)
					(2) edge (3)
					;
				\end{tikzpicture}\]
				\caption{  $P_{3} \odot T_{7}$ } 
			\end{subfigure}
			\caption{}\label{coro}
		\end{figure}
		\begin{Remark}\em{
	From definition of $X\odot H,$ it is clear that if $H$ is any graph and $X$ has     $s$ connected components, say $X_1,\dots, X_s,$  then $X\odot H$ has  again $s$ connected components $X_1\odot H,\dots, X_s\odot H.$ But if $X$ is a connected graph and $H$ has $l$ connected components then $X\odot H$ is still a connected graph. 
	It is obvious that $X\odot H$ has no isolated vertex even if $X$ and $H$ have isolated vertices. Therefore, we allow isolated vertices in both graphs $X$ and $H$. The isolated vertices of $H$ play a crucial role in our results. Therefore, we introduce the following terminologies.}
	\end{Remark}	
	
	%Let $X$ and $H$ be any two graphs, then from the Definition \ref{defcorona}, 

	%	such that they may consists of connected component or  disjoint connected components or may have isolated vertices.  In this case, we permit variables corresponding to isolated vertices of $X$ or $H$ in   $I(X)$ and $I(H),$ respectively. However, by Definition \ref{defcorona}, $X\odot H$ does not contain any isolated vertex. Therefore,  $I(X\odot H)$ may have connected component or disjoint connected components only.}
	
			For any graph $H,$ we denote the set of isolated vertices of $H$ by $i(H)$  and if $A:=V(H)\backslash i(H)$ then we denote the induced subgraph of $H$ on  $A$ by $H'.$ 
		%	Consider graph $G$ may have $a$ isolated vertices and some disjoint connected components on $b$ vertices (here $a$ and $b$ cannot be zero simultaneously). 
		It is easy to see that  $|V(H)|=|i(H)|+|A|$ and
		$I(H)=(x_ix_j,x_k : \{x_i,x_j\}\in E(H')\, \text{and}\, x_k\in i(H))$\big(for instance; if $E(H')=\emptyset$, then $I(H)=(x_k: x_k\in i(H))$\big). Also  $K[V(H)]/I(H)\cong K[V(H')]/I(H').$ Therefore, we have
$\depth(K[V(H)]/I(H))=\depth(K[V(H')]/I(H')$ and $\sdepth(K[V(H)]/I(H))=\sdepth(K[V(H')]/I(H')$. Similarly, we get  $\reg(K[V(H)]/I(H))=\reg(K[V(H')]/I(H')$ and $\dim(K[V(H)]/I(H))=\dim(K[V(H')]/I(H').$  If  $H$ is a null graph, then $|V(H)|=i(H)$ and  $\depth(K[V(H)]/I(H))=\sdepth(K[V(H)]/I(H))=\reg(K[V(H)]/I(H))=\dim(K[V(H)]/I(H))=0.$   Moreover, if $i(H)=\emptyset,$ then in this case we consider $K[i(H)]\cong K.$ For example, let $X=C_3$   and $H$ be a graph which is a union of two graphs, one is a complete graph on $4$ vertices and the other is a null graph on $3$ vertices, that is, $H=K_4 \cup N_3,$ as shown in Figure \ref{figcyclecorona}(A).   Figure \ref{figcyclecorona}(B) is the corona product of $C_3$ and $H.$  %For sake of convenience, we will use  these notions for $G$ throughout the work.
		\begin{figure}[H]
			\centering
			\begin{subfigure}[b]{0.45\textwidth}
				\centering
				\[\begin{tikzpicture}[x=0.9cm, y=0.8cm]
					\vertex[fill] (1) at (0,-3) [label=below:$y_{1}$] {};
					\vertex[fill] (2) at (1.5,-2) [label=below:$y_{2}$] {};
					\vertex[fill] (3) at (3,-3) [label=below:$y_{3}$] {};
					\vertex[fill] (x1) at (0,0) [label=below:$x_{1}$] {};
					\vertex[fill] (x2) at (0,1) [label=above:$x_{2}$] {};
					\vertex[fill] (x3) at (1,1) [label=above:$x_{3}$] {};
					\vertex[fill] (x4) at (1,0) [label=below:$x_{4}$] {};
					\vertex[fill] (x5) at (2,0) [label=below:$x_{5}$] {};
					\vertex[fill] (x6) at (3,0) [label=below:$x_{6}$] {};
					\vertex[fill] (x7) at (2.5,1) [label=below:$x_{7}$] {};
					\path 
					(x1) edge (x2)
					(x2) edge (x3)
					(x2) edge (x4)
					(x1) edge (x4)
					(x1) edge (x3)
					(x4) edge (x3)
					(1) edge (2)
					(1) edge (3)
					(2) edge (3)
					;
				\end{tikzpicture}\]
				\caption{ Graphs $C_3$ and $H$}\label{figcyclecorona1}
			\end{subfigure}
			\hfill
			\begin{subfigure}[b]{0.5\textwidth}
				\centering
				\[\begin{tikzpicture}[x=0.63cm, y=0.8cm]
					\vertex[fill] (1) at (0,-4) [label=below:$y_{1}$] {};
					\vertex[fill] (2) at (1.5,-3) [label=below:$y_{2}$] {};
					\vertex[fill] (3) at (3,-4) [label=below:$y_{3}$] {};
					\vertex[fill] (x1) at (0,0) [label=left:$x_{21}$] {};
					\vertex[fill] (x2) at (0,1) [label=above:$x_{22}$] {};
					\vertex[fill] (x3) at (1,1) [label=above:$x_{23}$] {};
					\vertex[fill] (x4) at (1,0) [label=below:$x_{24}$] {};
					\vertex[fill] (x5) at (2,0) [label=above:$x_{25}$] {};
					\vertex[fill] (x6) at (3,0) [label=above:$x_{26}$] {};
					\vertex[fill] (x7) at (2.5,1) [label=above:$x_{27}$] {};
					\vertex[fill] (u1) at (-4,-1.5) [label=below:$x_{11}$] {};
					\vertex[fill] (u2) at (-4,-0.5) [label=above:$x_{12}$] {};
					\vertex[fill] (u3) at (-3,-0.5) [label=above:$x_{13}$] {};
					\vertex[fill] (u4) at (-3,-1.5) [label=below:$x_{14}$] {};
					\vertex[fill] (u5) at (-2,-1.5) [label=above:$x_{15}$] {};
					\vertex[fill] (u6) at (-1,-1.5) [label=right:$x_{16}$] {};
					\vertex[fill] (u7) at (-1.5,-0.5) [label=right:$x_{17}$] {};
					\vertex[fill] (v1) at (4,-1.5) [label=left:$x_{31}$] {};
					\vertex[fill] (v2) at (4,-0.5) [label=above:$x_{32}$] {};
					\vertex[fill] (v3) at (5,-0.5) [label=above:$x_{33}$] {};
					\vertex[fill] (v4) at (5,-1.5) [label=below:$x_{34}$] {};
					\vertex[fill] (v5) at (6,-1.5) [label=above:$x_{35}$] {};
					\vertex[fill] (v6) at (7,-1.5) [label=above:$x_{36}$] {};
					\vertex[fill] (v7) at (6.5,-0.5) [label=above:$x_{37}$] {};
					\path 
					(x1) edge (x2)
					(x2) edge (x3)
					(x2) edge (x4)
					(x1) edge (x4)
					(x1) edge (x3)
					(x4) edge (x3)
					(u1) edge (u2)
					(u2) edge (u3)
					(u2) edge (u4)
					(u1) edge (u4)
					(u1) edge (u3)
					(u4) edge (u3)
					(v1) edge (v2)
					(v2) edge (v3)
					(v2) edge (v4)
					(v1) edge (v4)
					(v1) edge (v3)
					(v4) edge (v3)
					(1) edge (2)
					(1) edge (3)
					(2) edge (3)
					(2) edge (x2)
					(2) edge (x3)
					(2) edge (x4)
					(2) edge (x4)
					(2) edge (x3)
					(2) edge (x3)
					(1) edge (u2)
					(1) edge (u3)
					(1) edge (u4)
					(1) edge (u4)
					(1) edge (u3)
					(1) edge (u3)
					(3) edge (v2)
					(3) edge (v3)
					(3) edge (v4)
					(3) edge (v4)
					(3) edge (v3)
					(3) edge (v3)
					(x1) edge (2)
					(x2) edge (2)
					(x2) edge (2)
					(x1) edge (2)
					(x1) edge (2)
					(x4) edge (2)
					(u1) edge (1)
					(u2) edge (1)
					(u2) edge (1)
					(u1) edge (1)
					(u1) edge (1)
					(u4) edge (1)
					(v1) edge (3)
					(v2) edge (3)
					(v2) edge (3)
					(v1) edge (3)
					(v1) edge (3)
					(v4) edge (3)
					(v5) edge (3)
					(v6) edge (3)
					(v7) edge (3)
					(u5) edge (1)
					(u6) edge (1)
					(u7) edge (1)
					(x5) edge (2)
					(x6) edge (2)
					(x7) edge (2)
					;
				\end{tikzpicture}\]
				\caption{  $C_3 \odot H$ }\label{figcyclecorona2}
			\end{subfigure}
			\caption{ }\label{figcyclecorona}
		\end{figure}
		 For a monomial ideal $I$, $\supp{(I)}:=\{x_{i}:x_{i}|f\, \,  \text{for some} \,\,  f \in \mathcal{G}(I)\}.$
			
	\begin{Remark}
	{\em  We associate a graph $G_{I}$ to the  squarefree monomial ideal $I$ with $V(G_I)=\supp(I)$  and $E(G_I)=\{\{x_{i},x_{j}\}:x_{i}x_{j}\in \mathcal{G}(I)\}.$ Let $x_{t} \in S$ be a variable of the polynomial ring $S$ such that $x_{t} \notin I.$ Then $(I:x_{t})$ and  $(I,x_{t})$ are the squarefree monomial ideals of $S$ such that $G_{(I:x_{t})}$ and $G_{(I,x_{t})}$ are subgraphs of $G_{I}.$ Consider the graph $C_3\odot H$  as given in  Figure \ref{figcyclecorona}(B), if $S=K[V(C_3\odot H)],$  then	$G_{(I(C_3\odot H):y_{3})}$ and	$G_{(I(C_3\odot H),y_{3})}$ are subgraphs of $G_{I(C_3\odot H)}$ see Figure \ref{c4corona}. It is evident from the Figure \ref{c4corona}(A) and Figure \ref{c4corona}(B) that we have the following isomorphisms:
	\begin{equation*}
	    S/(I(C_3\odot H):y_3)\cong	\underset{l=1}{\overset{2}{\tensor_{K}}}K[V(K_{4})]/I(K_{4}) \tensor_{K }K[x_{15},x_{16},,x_{17},x_{25},x_{26},x_{27},y_{3}] 
	\end{equation*} and
	\begin{equation*}
	    S/(I(C_3\odot H),y_3)\cong	K[V(P_2\odot H)]/I(P_2\odot H)\tensor_{K} K[V(K_{4})]/I(K_{4})\tensor_{K} K[x_{35},x_{36},x_{37}].
	\end{equation*}}
\end{Remark}

\begin{figure}[H]
	\centering
	\begin{subfigure}[b]{0.45\textwidth}
	\centering
	\[\begin{tikzpicture}[x=0.6cm, y=0.6cm]
		\vertex[fill] (1) at (0,-4) [label=below:$y_{1}$] {};
		\vertex[fill] (2) at (1.5,-3) [label=below:$y_{2}$] {};
	%	\vertex[fill] (3) at (3,-4) [label=below:$y_{3}$] {};
		\vertex[fill] (x1) at (0,0) [label=below:$x_{21}$] {};
		\vertex[fill] (x2) at (0,1) [label=above:$x_{22}$] {};
		\vertex[fill] (x3) at (1,1) [label=above:$x_{23}$] {};
		\vertex[fill] (x4) at (1,0) [label=below:$x_{24}$] {};
	%	\vertex[fill] (x5) at (2,0) [label=below:$x_{25}$] {};
	%	\vertex[fill] (x6) at (3,0) [label=below:$x_{26}$] {};
	%	\vertex[fill] (x7) at (2.5,1) [label=above:$x_{27}$] {};
		\vertex[fill] (u1) at (-4,-1.5) [label=below:$x_{11}$] {};
		\vertex[fill] (u2) at (-4,-0.5) [label=above:$x_{12}$] {};
		\vertex[fill] (u3) at (-3,-0.5) [label=above:$x_{13}$] {};
		\vertex[fill] (u4) at (-3,-1.5) [label=below:$x_{14}$] {};
		%\vertex[fill] (u5) at (-2,-1.5) [label=below:$x_{15}$] {};
		%\vertex[fill] (u6) at (-1,-1.5) [label=below:$x_{16}$] {};
	%	\vertex[fill] (u7) at (-1.5,-0.5) [label=above:$x_{17}$] {};
		\vertex[fill] (v1) at (4,-1.5) [label=below:$x_{31}$] {};
		\vertex[fill] (v2) at (4,-0.5) [label=above:$x_{32}$] {};
		\vertex[fill] (v3) at (5,-0.5) [label=above:$x_{33}$] {};
		\vertex[fill] (v4) at (5,-1.5) [label=below:$x_{34}$] {};
		\vertex[fill] (v5) at (6,-1.5) [label=below:$x_{35}$] {};
		\vertex[fill] (v6) at (7,-1.5) [label=below:$x_{36}$] {};
		\vertex[fill] (v7) at (6.5,-0.5) [label=above:$x_{37}$] {};
		\path 
		(x1) edge (x2)
		(x2) edge (x3)
		(x2) edge (x4)
		(x1) edge (x4)
		(x1) edge (x3)
		(x4) edge (x3)
		(u1) edge (u2)
		(u2) edge (u3)
		(u2) edge (u4)
		(u1) edge (u4)
		(u1) edge (u3)
		(u4) edge (u3)

		;
	\end{tikzpicture}\] 
		\caption{$G_{(I(C_3\odot H):y_{3})}$}\label{colon}
	\end{subfigure}
	\hfill
	\begin{subfigure}[b]{0.45\textwidth}
\centering
\[\begin{tikzpicture}[x=0.55cm, y=0.6cm]
	\vertex[fill] (1) at (0,-4) [label=below:$y_{1}$] {};
	\vertex[fill] (2) at (1.5,-3) [label=below:$y_{2}$] {};
	\vertex[fill] (3) at (3,-4) [label=below:$y_{3}$] {};
	\vertex[fill] (x1) at (0,0) [label=below:$x_{21}$] {};
	\vertex[fill] (x2) at (0,1) [label=above:$x_{22}$] {};
	\vertex[fill] (x3) at (1,1) [label=above:$x_{23}$] {};
	\vertex[fill] (x4) at (1,0) [label=below:$x_{24}$] {};
	\vertex[fill] (x5) at (2,0) [label=above:$x_{25}$] {};
	\vertex[fill] (x6) at (3,0) [label=above:$x_{26}$] {};
	\vertex[fill] (x7) at (2.5,1) [label=above:$x_{27}$] {};
	\vertex[fill] (u1) at (-4,-1.5) [label=below:$x_{11}$] {};
	\vertex[fill] (u2) at (-4,-0.5) [label=above:$x_{12}$] {};
	\vertex[fill] (u3) at (-3,-0.5) [label=above:$x_{13}$] {};
	\vertex[fill] (u4) at (-3,-1.5) [label=below:$x_{14}$] {};
	\vertex[fill] (u5) at (-2,-1.5) [label=below:$x_{15}$] {};
	\vertex[fill] (u6) at (-1,-1.5) [label=right:$x_{16}$] {};
	\vertex[fill] (u7) at (-1.5,-0.5) [label=above:$x_{17}$] {};
	\vertex[fill] (v1) at (4,-1.5) [label=below:$x_{31}$] {};
	\vertex[fill] (v2) at (4,-0.5) [label=above:$x_{32}$] {};
	\vertex[fill] (v3) at (5,-0.5) [label=above:$x_{33}$] {};
	\vertex[fill] (v4) at (5,-1.5) [label=below:$x_{34}$] {};
%	\vertex[fill] (v5) at (6,-1.5) [label=below:$x_{35}$] {};
%	\vertex[fill] (v6) at (7,-1.5) [label=below:$x_{36}$] {};
%	\vertex[fill] (v7) at (6.5,-0.5) [label=above:$x_{37}$] {};
	\path 
	(x1) edge (x2)
	(x2) edge (x3)
	(x2) edge (x4)
	(x1) edge (x4)
	(x1) edge (x3)
	(x4) edge (x3)
	(u1) edge (u2)
	(u2) edge (u3)
	(u2) edge (u4)
	(u1) edge (u4)
	(u1) edge (u3)
	(u4) edge (u3)
	(v1) edge (v2)
	(v2) edge (v3)
	(v2) edge (v4)
	(v1) edge (v4)
	(v1) edge (v3)
	(v4) edge (v3)
	(1) edge (2)

(2) edge (x2)
(2) edge (x3)
(2) edge (x4)
(2) edge (x4)
(2) edge (x3)
(2) edge (x3)
(1) edge (u2)
(1) edge (u3)
(1) edge (u4)
(1) edge (u4)
(1) edge (u3)
(1) edge (u3)

	(x1) edge (2)
	(x2) edge (2)
	(x2) edge (2)
	(x1) edge (2)
	(x1) edge (2)
	(x4) edge (2)
	(u1) edge (1)
	(u2) edge (1)
	(u2) edge (1)
	(u1) edge (1)
	(u1) edge (1)
	(u4) edge (1)

	(u5) edge (1)
	(u6) edge (1)
	(u7) edge (1)
	(x5) edge (2)
	(x6) edge (2)
	(x7) edge (2)
	;
\end{tikzpicture}\] 	\caption{ $G_{(I(C_3\odot H),y_{3})}$}\label{coma}
	\end{subfigure}
\caption{}\label{c4corona}
\end{figure} 

%--------------------

	\section{Depth, Stanley depth and Regularity of residue class rings of some edge ideals}\label{sec2}
	%of corona product of some graphs with any graph G we compute the values of depth and  regularity
Let  $P_n,C_n, K_n, S_n$ and $K_{u,v}$ are the graphs as defined earlier and  $H$ be any graph. %For sake of convenience, we denote  $\mathbb{K}_{1}:=N_{1}\odot H,$ $\mathbb{P}_{n}:=P_{n}\odot H,$ $\mathbb{C}_{n}=C_{n}\odot H,$ $\mathbb{K}_{n}=K_{n}\odot H,$ $\mathbb{S}_{n}=S _{n}\odot H$ and $\mathbb{K}_{u,v}=K_{u,v}\odot H.$ 
In this section, we focus on the invariants depth, Stanley depth, projective dimension and  regularity of $S/I(X\odot H),$ where $X\in \{P_n,C_n, K_n, S_n,K_{u,v}\}.$ 
%$S/I(N_{1}\odot H),S/I(P_{n}\odot H),S/I(C_{n}\odot H),S/I(K_{n}\odot H),S/I(S_{n}\odot H)$ and $S/I(K_{u,v}\odot H).$
\begin{Remark}\em{
If $H$ is a null graph and $X\in \{P_n,C_n, K_n, S_n,K_{u,v}\},$  then $X\odot H$
%$N_{1}\odot H, P_{n}\odot H,C_{n}\odot H,K_{n}\odot H,S _{n}\odot H$ and $K_{u,v}\odot H$ 
is a $|V(H)|$-fold bristled graph of $X$ and  we denote it by $Br_{|V(H)|}(X).$}
\end{Remark}
\noindent This section is further divided into two subsections. In the first subsection, we  discuss depth, Stanley depth and projective dimension of $S/I(X\odot H)$, while in the second subsection, we  discuss the regularity.

%We will also discuss the  Stanley's inequality and projective dimension. 
	
	%	\subsection{Corona product of Path and cycle with any graph G}
	%and $K_{m}$ be the complete graph on $m$ vertices $\{w_{1},w_{2},\dots , w_{m}\}$ and edge set $\{w_{k}w_{l}: 1\leq k<l\leq m\}.$ Let $\overline{K}_{m}$ denotes the compliment of a complete graph $K_{m}$. For simplicity we assume $\mathbf{P}_{n,m}=P_{n}\odot \overline{K}_{m},$ and $\mathbf{C}_{n,m}=C_{n}\odot \overline{K}_{m}.$ If we take the corona product of $P_{n}$ with $\overline{K}_{m}$ or $C_{n}$ with $\overline{K}_{m},$ then we assume that the vertices of $c$th copy of $\overline{K}_{m}$ are $\{w_{c1},w_{c2},\dots , w_{cm}\}$. As $|V(P_{n}\odot \overline{K}_{m})|=|V(C_{n}\odot \overline{K}_{m})|=n(m+1)$, so we have $S_{n(m+1)}=K[V(P_{n}),\cup^{n}_{c=1}V(\overline{K}_{m}^{c})]=K[V(C_{n}),\cup ^{n}_{c=1}V(\overline{K}_{m}^{c})].$
	%	The minimal generating sets for the edge ideals of $\mathbf{P}_{n,m}$ and $\mathbf{C}_{n,m}$ are:
	%\begin{eqnarray*}
	%\mathcal{G}(I(\mathbf{P}_{n,m}))=\big(\cup^{n}_{j=1}\{y_{j}w_{j1},y_{j}w_{j2}, \dots ,y_{j}w_{jm} \}\cup^{n-1}_{j=1}\{y_{j}y_{j+1}\}\big),\\
	%\mathcal{G}(I(\mathbf{C}_{n,m}))=\big(I(\mathbf{P}_{n,m})\cup \{y_{n}y_{1}\}\big).
	%	\end{eqnarray*}
	%	and $P_{1}$ be a trivial graph with vertex $y_{1}.$ % For $n\geq2,$ $P_{n}$ be the path on $n$ vertices $\{y_{1},y_{2},\dots,y_{n}\},$ and for $n\geq3,$ $C_{n}$ be the cycle on $n$ vertices $\{y_{1},y_{2},\dots,y_{n}\}.$
%	Let $N_{1}=\{y_{1}\}$ be a trivial graph and we denote  $\mathbb{K}_{1}:=N_{1}\odot H$. 
\subsection{Depth, Stanley depth and projective dimension}\label{sec2.1}
\paragraph{} \bigskip If $N_1$ denotes a trivial graph, 
%The minimal generating set for the edge ideal of $N_1 \odot H$ is: \begin{eqnarray*}	\mathcal{G}(I(N_1 \odot H))=\big(\{y_{1}x_{11}, \dots ,y_{1}x_{1(|i(H)|+|A|)}\}\cup\{x_{1l}x_{1k}\in E(H_{1})| ~ \text{iff}~ x_{l}x_{k}\in E(H) \}\big).\end{eqnarray*} 
	then  we have the following $K$-algebra isomorphisms: \begin{equation}\label{vertexcolon}
			K[V(N_1 \odot H)]/(I(N_1 \odot H):y_{1})  \cong K[y_{1}],
		\end{equation} 
		\begin{equation}\label{vertexcoma}
			K[V(N_1 \odot H)]/(I(N_1 \odot H),y_{1})\cong K[V(H)]/I(H) \tensor_{K} K[i(H)].
		\end{equation} Now,  we give an elementary lemma that will be used frequently in  this subsection.
	\begin{Lemma}\label{triv}
		Let $ S=K[V(N_1 \odot H)]$. Then % If  $\depth(K[V(H)]/I(H))= t$, then
		$$	\depth(S/I(N_1 \odot H))= 	\sdepth(S/I(N_1 \odot H))=1.$$   	\end{Lemma}
	\begin{proof}
		First we will prove the result for depth.	If $H$ is a null graph,	then we have  $|V(H)|=|i(H)|$ and 
		 $N_1 \odot H\cong S_{|V(H)|}.$  Thus by Lemma \ref{leAli},  $\depth(S/I(N_1 \odot H))=1$. % Let $S''=K[V(\mathbb{K}_{1}/y_{1})]$ and $S'=K[V(\mathbb{K}_{1} \backslash y_{1})].$   
		Let $H$ is not a null graph, %We have $t\geq 1$ in this case. 	
	 by using  Eq. \ref{vertexcolon}, we get $\depth(S/(I(N_1 \odot H):y_{1}))=\depth(K[y_1])=1.$  
		By using Lemma \ref{le3} on  Eq.  \ref{vertexcoma},  we have
		$\depth(S/(I(N_1 \odot H),y_{1}))= \depth(K[V(H)]/I(H))+|i(H)|\geq 1.$ Thus by Lemma \ref{exacther}, $\depth(S/I(N_1 \odot H))=1$. The proof of Stanley depth is similar as depth by using Lemma \ref{exacthersdepth} in place of Lemma \ref{exacther}.   
	\end{proof}
	\begin{Corollary}
		If $S=K[V(N_1 \odot H)],$ then	 $\pdim(S/I(N_1 \odot H))= |V(H)|.$
	\end{Corollary}
	\begin{proof}
		By using Lemma \ref{auss13} %{\cite[Theorems 1.3.3]{depth}} 
		and Lemma \ref{triv}, the required result follows. 
	\end{proof}
	\begin{Remark}\label{zeroandtriv}
\em{ For $n=0,$ we define $I(P_0\odot H)=(0),$ we have  $K[V(P_0\odot H)]/I(P_0\odot H)\cong K.$ Thus $\depth(K[V(P_0\odot H)]/I(P_0\odot H))=\sdepth(K[V(P_0\odot H)]/I(P_0\odot H))=0.$ Moreover, if $n=1,$ then $P_1\odot H\cong N_1\odot H.$}
	\end{Remark}
	
	\noindent		Let $y_{n}$ be a leaf of path $P_n,$ we have the following $K$-algebra isomorphisms:	\begin{equation}\label{path1} K[V(P_n \odot H)]/(I(P_n \odot H):y_{n})\cong K[V(P_{n-2} \odot H)]/I(P_{n-2} \odot H))\tensor_{K} K[V(H)]/I(H)\tensor_{K} K[i(H)\cup \{y_{n}\}],
		\end{equation}
		\begin{equation}\label{path2}	K[V(P_n \odot H)]/(I(P_n \odot H),y_{n})\cong K[V(P_{n-1} \odot H)]/I(P_{n-1}\odot H)\tensor_{K} K[V(H)]/I(H)\tensor_{K} K[i(H)].
		\end{equation} 
	\begin{Theorem}\label{The1} Let $n\geq 1$ and $ S=K[V(P_n \odot H)]$. Then 
		
		\begin{itemize} 
			\item[(a)] 	$\depth(S/I(P_n \odot H))=\big\lceil\frac{n}{2}\big\rceil+\big\lceil\frac{n-1}{2}\big\rceil \big(\depth(K[V(H)]/I(H))+|i(H)|\big).$ 
			\item[(b)] $\sdepth(S/I(P_n \odot H))\geq \big\lceil\frac{n}{2}\big\rceil+\big\lceil\frac{n-1}{2}\big\rceil \big(\sdepth(K[V(H)]/I(H))+|i(H)|\big).$

		\noindent In particular,	if $H$ is a null graph, then	 $$\sdepth(S/I(P_n \odot H))= \big\lceil\frac{n}{2}\big\rceil+\big\lceil\frac{n-1}{2}\big\rceil |V(H)|.$$	  
			
	%	\noindent Moreover, if $H$ is a null graph, then	 $\sdepth(S/I((P_n \odot H))= \big\lceil\frac{n}{2}\big\rceil+\big\lceil\frac{n-1}{2}\big\rceil |i(H)|.$ 
		\end{itemize}

	\end{Theorem}
	\begin{proof}

	Let  $t:=\depth(K[V(H)]/I(H)).$	First we will prove the result for depth by using induction on $n.$   If $n=1,$ then $P_1\odot H\cong N_1\odot H$ and the result follows  by Lemma \ref{triv}. %	If $n=2,3,$ by Eqs. \ref{path1} and \ref{path2},
	%	Lemmas \ref{le3}, \ref{LEMMA1.5} and  \ref{triv}, we have $\depth(S''/I(\mathbb{P}_{2}/y_{2})[\overline{Q}_{\mathbb{P}_{2}}(y_2)])=t+|i(H)|+1,$ 	$\depth(S'/I(\mathbb{P}_{2}\backslash y_2)[P_{\mathbb{P}_{2}}(y_{2})])= t+|i(H)|+1,$ $\depth(S''/I(\mathbb{P}_{3}/y_{3})[\overline{Q}_{\mathbb{P}_{3}}(y_3)])=t+|i(H)|+2,$ and	$\depth(S'/I(\mathbb{P}_{3}\backslash y_3)[P_{\mathbb{P}_{3}}(y_{3})])= 2t+2|i(H)|+1.$  Thus by using Remark \ref{remaar}, we get $\depth(S/I(\mathbb{P}_{2}))=1+t+|i(H)|$	and  $\depth(S/I(\mathbb{P}_{3}))=2+t+|i(H)|.$ 	
	Let $n\geq 2$.	By induction on $n$ and using  Eq. \ref{path1}, Eq. \ref{path2}, Lemma \ref{le3}, Lemma \ref{LEMMA1.5}, we get
 \begin{equation*}
     \begin{split}
         &\depth (S/(I(P_n \odot H):y_{n}))\\&\quad=\depth(K[V(P_{n-2} \odot H)]/I(P_{n-2} \odot H))+ \depth(K[V(H)]/I(H)) +\depth(K[i(H)\cup \{y_{n}\}])\\ &\quad=\big\lceil\frac{n-2}{2}\big\rceil+\big\lceil\frac{n-3}{2}\big\rceil \big(t+|i(H)|\big)+t+|i(H)|+1= \big\lceil\frac{n}{2}\rceil+\big\lceil\frac{n-1}{2}\big\rceil (t+|i(H)|),
     \end{split}
 \end{equation*}
 \begin{equation*}
     \begin{split}
         &\depth (S/(I(P_n \odot H),y_{n}))\\&\quad=\depth(K[V(P_{n-1} \odot H)]/I(P_{n-1}\odot H))+ \depth(K[V(H)]/I(H))+ \depth(K[i(H)])\\&\quad =\big\lceil\frac{n-1}{2}\big\rceil+\big\lceil\frac{n-2}{2}\big\rceil\big(t+|i(H)|\big) +t+|i(H)|= \big\lceil\frac{n-1}{2}\big\rceil+\big\lceil\frac{n}{2}\big\rceil (t+|i(H)|).
     \end{split}
 \end{equation*}
	Since $\big\lceil\frac{n}{2}\rceil+\big\lceil\frac{n-1}{2}\big\rceil (t+|i(H)|)\leq\big\lceil\frac{n-1}{2}\big\rceil+\big\lceil\frac{n}{2}\big\rceil (t+|i(H)|),$	thus by using Lemma \ref{exacther}, we get $\depth(S/(I(P_n \odot H))= \lceil\frac{n}{2}\rceil+\lceil\frac{n-1}{2}\rceil (t+|i(H)|). $ 
	
	To prove the result for Stanley depth, consider if $H$ is not a null graph, then by considering  Eq. \ref{path1} and Eq. \ref{path2} and applying Lemma \ref{le3}, Lemma \ref{LEMMA1.5} and Lemma \ref{rle2}, we get the required inequality for Stanley depth.  But if   $H$ is a null graph, we have \begin{equation*} S/(I(P_n \odot H):y_{n})\cong K[V(P_{n-2} \odot H)]/I(P_{n-2} \odot H))\tensor_{K} K[V(H)\cup \{y_{n}\}],
		\end{equation*}
		\begin{equation*}	S/(I(P_n \odot H),y_{n})\cong K[V(P_{n-1} \odot H)]/I(P_{n-1}\odot H)\tensor_{K} K[V(H)],
		\end{equation*} and the proof is similar to depth by replacing Lemma \ref{exacther} by Lemma \ref{exacthersdepth}.
		
		%\end{description}
		
	\end{proof}

	\begin{Corollary}
			If	Stanley's inequality holds for $K[V(H)]/I(H),$ then it also holds  for $S/I(P_n \odot H).$
	\end{Corollary}
	%	\begin{Remark}\label{b3} \em{ If $H=i(H),$ then   $ \depth(S/I(\mathbb{P}_{n}))=\sdepth(S/I(\mathbb{P}_{n}))=\big\lceil\frac{n}{2}\big\rceil+\big\lceil\frac{n-1}{2}\big\rceil |i(H)|$ by using the similar arguments as in Theorem \ref{The1}.} 	\end{Remark}
	
	%\begin{Corollary}  Let $n\geq 2$ and $I(\mathbb{P}_{n})\subset S:=K[V(\mathbb{P}_{n})].$ If $H=i(H)$, then $$ \depth(S/I(\mathbb{P}_{n}))=\sdepth(S/I(\mathbb{P}_{n}))=\big\lceil\frac{n}{2}\big\rceil+\big\lceil\frac{n-1}{2}\big\rceil |i(H)|.$$ Moreover, if  $\depth(K[V(H)]/I(H))=t\geq 0$, then $$		\pdim(S/I(\mathbb{P}_{n}))= 	n(|i(H)|+|A|+1)-	\big\lceil\frac{n}{2}\big\rceil-\big\lceil\frac{n-1}{2}\big\rceil (t+|i(H)|).$$ 	\end{Corollary}
	
	%	\begin{Corollary}  Let $n\geq 2$ and $I(\mathbb{P}_{n})\subset S:=K[V(\mathbb{P}_{n})].$  Then  \begin{itemize}     \item [(a)] If  $H=i(H)$, then	  $ \depth(S/I(\mathbb{P}_{n}))=\sdepth(S/I(\mathbb{P}_{n}))=\big\lceil\frac{n}{2}\big\rceil+\big\lceil\frac{n-1}{2}\big\rceil |i(H)|.$ \item [(b)] $		\pdim(S/I(\mathbb{P}_{n}))= 	n(|i(H)|+|A|+1)-	\big\lceil\frac{n}{2}\big\rceil-\big\lceil\frac{n-1}{2}\big\rceil (t+|i(H)|).$  \end{itemize}	\end{Corollary}
	
	%	\begin{Corollary} 	Let $n\geq 2$ and $I(\mathbb{P}_{n})\subset S:=K[V(\mathbb{P}_{n})]$. If  $H=i(H)$, then	$$ \depth(S/I(\mathbb{P}_{n}))=\sdepth(S/I(\mathbb{P}_{n}))=\big\lceil\frac{n}{2}\big\rceil+\big\lceil\frac{n-1}{2}\big\rceil |i(H)|.$$	\end{Corollary}	\begin{proof}  If $H$ is a null graph then $t=0$ in this case and the result follows by using similar arguments of Theorem \ref{The1} .\end{proof}
	\begin{Corollary} \label{corsub}	If $n\geq 1$, then  
		%and $\depth(K[V(H)]/I(H))=t$. If $H=N_{a}$ for $a\geq 1$, then
		%	\begin{equation*}	\pdim(S/I(\mathbb{P}_{n}))= 	n(a+1)-	n-(a-1)\bigg\lceil\frac{n-1}{2}\bigg\rceil.	\end{equation*}
		%	If  $H$ is a graph such that $a\geq 0$ and  $b\geq 2$, then 
		$$		\pdim(S/I(P_n \odot H))= 	n(|V(H)|+1)-	\big\lceil\frac{n}{2}\big\rceil-\big\lceil\frac{n-1}{2}\big\rceil \big(\depth(K[V(H)]/I(H)\big)+|i(H)|).$$
	\end{Corollary}
	\begin{proof}
		As $|V(P_n \odot H)|=n(|V(H)|+1)$.	By using Lemma \ref{auss13} and Theorem \ref{The1}, one can  find the required result.
	\end{proof}
	\begin{Corollary}
		    Let $n,m,s\geq 1$ and $ S=K[V(P_n \odot P_m)]$. Then 
		
		\begin{itemize} 
			\item[(a)] 	$\depth(S/I(P_n \odot P_m))=\big\lceil\frac{n}{2}\big\rceil+\big\lceil\frac{n-1}{2}\big\rceil\big\lceil\frac{m}{3}\big\rceil .$ 
				\item[(b)]  $	\pdim(S/I(P_n \odot P_m))= 	n(m+1)-	\big\lceil\frac{n}{2}\big\rceil-\big\lceil\frac{n-1}{2}\big\rceil\big\lceil\frac{m}{3}\big\rceil .$
			%\item[(c)] 	  $\sdepth(S/I(P_n \odot P_m))\geq \big\lceil\frac{n}{2}\big\rceil+\big\lceil\frac{n-1}{2}\big\rceil\big\lceil\frac{m}{3}\big\rceil.$ 
			\item[(c)]$\depth(K[V(Br_s(P_n))]/I(Br_s(P_n)))=\sdepth(K[V(Br_s(P_n))]/I(Br_s(P_n)))=\big\lceil\frac{n}{2}\big\rceil+\big\lceil\frac{n-1}{2}\big\rceil s.$ 
	%	\noindent Moreover, if $H$ is a null graph, then	 $\sdepth(S/I((P_n \odot H))= \big\lceil\frac{n}{2}\big\rceil+\big\lceil\frac{n-1}{2}\big\rceil |i(H)|.$ 
		\end{itemize}
		\end{Corollary}

\noindent% Let $V(C_n)=\{y_1,y_2,\dots, y_n\}$ and $E(C_{n})=\{\{y_{j},y_{j+1}\}:1\leq j\leq n-1\}\cup \{y_{1},y_{n}\}.$ 	%The minimal generating set for the edge ideal of $C_n \odot H$ is: 	\begin{equation*}	\mathcal{G}(I(C_n \odot H))=\mathcal{G}(I(P_n \odot H))\cup \{y_{n}y_{1}\}.	\end{equation*}
	For $y_{n-1}\in V(C_n \odot H),$ we have the following $K$-algebra isomorphisms:
		\begin{multline}\label{b1}
		   K[V(C_n \odot H)]/(I(C_n \odot H):y_{n-1})\cong\\ K[V(P_{n-3} \odot H)]/I(P_{n-3} \odot H) \underset{l=1}{\overset{2}{\tensor_{K}}} K[V(H)]/I(H) \underset{l=1}{\overset{2}{\tensor_{K}}}K[i(H)] \tensor_{K} K[y_{n-1}],
		\end{multline}
	
		\begin{equation}\label{b2}
			K[V(C_n \odot H)]/(I(C_n \odot H),y_{n-1})\cong K[V(P_{n-1}\odot H)]/I(P_{n-1}\odot H)\tensor_{K} K[V(H)]/I(H) \tensor_{K} K[i(H)].
		\end{equation}   
	\begin{Theorem}\label{The3}
		If $n\geq 3$ and  $ S=K[V(C_n \odot H)],$ then
		\begin{itemize}
			\item[(a)] $\depth(S/I(C_n \odot H))= \big\lceil\frac{n-1}{2}\big\rceil+\big\lceil\frac{n}{2}\big\rceil \big(\depth(K[V(H)]/I(H))+|i(H)|\big).$ 
			\item[(b)] $\sdepth(S/I(C_n \odot H))\geq \big\lceil\frac{n-1}{2}\big\rceil+\big\lceil\frac{n}{2}\big\rceil \big(\sdepth(K[V(H)]/I(H))+|i(H)|\big).$ 	 
			
			\noindent In particular, if $H$ is a null graph, then $$\sdepth(S/I(C_n \odot H))= \big\lceil\frac{n-1}{2}\big\rceil+\big\lceil\frac{n}{2}\big\rceil |V(H)|.$$	
		\end{itemize}

		%			Let $C_{n}$ be a cycle on $n\geq3$ vertices and $a\geq 1$, if $H$ is a null graph then 
		%			%		\begin{eqnarray*}
		%			%		\depth(S_{n(q+1)}/I(C_{n}\odot H))= \left\{
		%			%			\begin{array}{ll}
		%			%				\frac{n(a+1)}{2}, & \hbox{if $n$ is even;} \\\\
		%			%				\frac{(n-1)(a+1)}{2}+a, & \hbox{if $n$ is odd;}\\\\ \lceil\frac{n-1}{2}\rceil+\lceil\frac{n}{2}\rceil t& \hbox{otherwise.}
		%			%			\end{array}\right.	\end{eqnarray*}
		%			
		%			\begin{equation*}
		%				\depth(S/I(\mathbb{C}_{n}))= n+(a-1)\bigg\lceil\frac{n}{2}\bigg\rceil,	\end{equation*}
		%			otherwise, $$	\depth(S/I(\mathbb{C}_{n}))= \big\lceil\frac{n-1}{2}\big\rceil+\big\lceil\frac{n}{2}\big\rceil t.$$
	\end{Theorem}
	\begin{proof}
	Firstly, we discuss the result for depth. Let  $t:=\depth(K[V(H)]/I(H)).$		By using  Theorem \ref{The1}, Lemma \ref{le3} and Lemma \ref{LEMMA1.5} on Eq. \ref{b1} and Eq. \ref{b2}, we get
 \begin{equation*}
     \begin{split}
         &\depth(S/(I(C_n \odot H):y_{n-1}))\\&\quad=\depth(K[V(P_{n-3} \odot H)]/I(P_{n-3} \odot H)) +2\depth( K[V(H)]/I(H)) +2\depth(K[i(H)])\\&\quad \quad+ \depth(K[y_{n-1}])\\&\quad=\lceil\frac{n-3}{2}\rceil+\lceil\frac{n-4}{2}\rceil \big(t+|i(H)|\big)+2t+2|i(H)|+1= \lceil\frac{n-1}{2}\rceil+\lceil\frac{n}{2}\rceil (t+|i(H)|),
     \end{split}
 \end{equation*} \begin{equation*}
     \begin{split}
         &\depth(S/(I(C_n \odot H),y_{n-1}))\\&\quad =\depth(K[V(P_{n-1}\odot H)]/I(P_{n-1}\odot H))+ \depth(K[V(H)]/I(H)) +\depth( K[i(H)])\\&\quad= \lceil\frac{n-1}{2}\rceil+\lceil\frac{n-2}{2}\rceil \big(t+|i(H)|\big)+t+|i(H)|\\&\quad= \lceil\frac{n-1}{2}\rceil+\lceil\frac{n}{2}\rceil (t+|i(H)|).
     \end{split}
 \end{equation*}
		Using Lemma \ref{exacther}, completes the proof for depth. 
		
		Now, we prove the result  for Stanley depth. If $H$ is not a null graph, then we get the required lower bound for Stanley depth by applying Lemma \ref{le3}, Lemma \ref{LEMMA1.5}, Theorem \ref{The1} and Lemma \ref{rle2} on Eq. \ref{b1} and Eq. \ref{b2}.	But if $H$ is a null graph, then  by  Eq. \ref{b1} and Eq. \ref{b2} we have the following $K$-algebra isomorphisms: 
  \begin{equation*}
       S/(I(C_n \odot H):y_{n-1})\cong K[V(P_{n-3} \odot H)]/I(P_{n-3} \odot H)  \underset{l=1}{\overset{2}{\tensor_{K}}}K[V(H)] \tensor_{K} K[y_{n-1}],
  \end{equation*}
		\begin{equation*}
			S/(I(C_n \odot H),y_{n-1})\cong K[V(P_{n-1}\odot H)]/I(P_{n-1}\odot H)\tensor_{K} K[V(H)],
		\end{equation*} and the proof  is similar to depth by using Lemma \ref{exacthersdepth} instead of Lemma \ref{exacther}.  
		%		If $ \mathbb{C}_n$ is $|i(H)|$-fold bristled graph then $t=0$ and result obtain by using the similar arguments.
		%	\end{description}
	\end{proof}%\noindent Stanley's inequality holds for $S/I(\mathbb{C}_{n})$ if it holds for  $K[V(H)]/I(H).$
	\begin{Corollary}
	 Stanley's inequality holds for $S/I(C_n \odot H)$ if it holds for  $K[V(H)]/I(H).$
	\end{Corollary} 
	%\begin{Remark}	\em{If $H=i(H),$ then 	$\depth(S/I(\mathbb{C}_{n}))=\sdepth(S/I(\mathbb{C}_{n}))=\big\lceil\frac{n-1}{2}\big\rceil+\big\lceil\frac{n}{2}\big\rceil|i(H)|$ by using similar strategy of Theorem \ref{The3}.}\end{Remark}
	\begin{Corollary} \label{corsub}
		Let $n\geq 3.$ Then

		%	If $S_{n(q+1)}=K[V(C_{n}),\cup ^{n}_{c=1}V(H_{c})]$ and  $I(C_{n}\odot H)\subset S$, then
		%		\begin{eqnarray*}
		%		\pdim(S_{n(q+1)}/I(C_{n}\odot H))=\left\{
		%		\begin{array}{ll}
		%		n(a+1)-	\frac{n(a+1)}{2}, & \hbox{if $n$ is even and $b= 0$;} \\\\
		%		n(a+1)-	\frac{(n-1)(a+1)}{2}-a, & \hbox{if $n$ is odd and $b= 0$;}\\\\ n(q+1)-\lceil\frac{n-1}{2}\rceil+\lceil\frac{n}{2}\rceil t& \hbox{otherwise.}
		%		\end{array}\right.	\end{eqnarray*}
		
		\begin{equation*}
			\pdim(S/I(C_n \odot H))= 	n(|V(H)|+1)-	\big\lceil\frac{n-1}{2}\big\rceil-\big\lceil\frac{n}{2}\big\rceil \big(\depth(K[V(H)]/I(H))+|i(H)|\big).	\end{equation*}
		
	\end{Corollary}
	\begin{proof}
		As $|V(C_n \odot H)|=n(|V(H)|+1)$.	By using Lemma \ref{auss13} and Theorem \ref{The3}, one can  find the required result.
	\end{proof}
	
	%	\begin{Corollary}	Let $n\geq 3$ and $S:=K[V(\mathbb{C}_{n})]$ be a polynomial ring. Stanley's inequality holds for $S/I(\mathbb{C}_{n})$ if it holds for  $K[V(H)]/I(H).$	\end{Corollary}

%	Let $K_{n}$ be the complete graph on $n$ vertices say $V(H)=\{y_{1},\dots , y_{n}\}.$ %The minimal generating set for the edge ideal of 
%	Denote $\mathbb{K}_{n}=K_{n}\odot H$, where $H$ be any graph.
	\begin{Corollary}
		Let $n,m\geq3$, $s\geq1$ and $S=K[V(C_n \odot C_m)].$ Then
		\begin{itemize}
			\item[(a)] $\depth(S/I(C_n \odot C_m))= \big\lceil\frac{n-1}{2}\big\rceil+\big\lceil\frac{n}{2}\big\rceil \big\lceil\frac{m-1}{3}\big\rceil.$ 
		%	\item[(b)] 	 $\sdepth(S/I(C_n \odot C_m))\geq \big\lceil\frac{n-1}{2}\big\rceil+\big\lceil\frac{n}{2}\big\rceil \big\lceil\frac{m-1}{3}\big\rceil.$ 
			\item[(b)] $		\pdim(S/I(C_n \odot C_m))= n(m+1)-\big\lceil\frac{n-1}{2}\big\rceil-\big\lceil\frac{n}{2}\big\rceil \big\lceil\frac{m-1}{3}\big\rceil.$
			\item[(c)]$\depth(K[V(Br_s(C_n))]/I(Br_s(C_n)))=\sdepth(K[V(Br_s(C_n))]/I(Br_s(C_n)))=\big\lceil\frac{n-1}{2}\big\rceil-\big\lceil\frac{n}{2}\rceil s.$ 
		\end{itemize}
		
	\end{Corollary}
\noindent %Let $V(K_n)= \{y_{1},y_{2},\dots , y_{n}\}$ and edge set $E(K_n)=\{\{y_{k},y_{l}\}: 1\leq k<l\leq n\}.$
%	The minimal generating set for the edge ideal of $K_{n}\odot H$ is as follows:\begin{multline*}	\mathcal{G}(I(K_{n}\odot H))=\big(\cup^{n}_{j=1}\{y_{j}x_{j1}, \dots ,y_{j}x_{j(i(H)+|A|)}\}\cup^{n}_{j=1}\{x_{jl}x_{jk}\in E(H_{j})| ~ \text{iff}~ x_{l}x_{k}\in E(H)) \}\\ \cup\{y_{d}y_{e}: 1\leq d<e\leq n: d,e\in V(K_n) \}\big).	\end{multline*}
		For $y_{j}\in V(K_{n}\odot H),$ where $y_{j}$ is an arbitrary vertex of $K_n$, we get the following isomorphisms:  
		\begin{equation}\label{c1}
			K[V(K_{n}\odot H)]/(I(K_{n}\odot H):y_{j})\cong  \underset{l=1}{\overset{n-1}{\tensor_K}} K[V(H)]/I(H) \underset{l=1}{\overset{n-1}{\tensor_K}}K[i(H)] \tensor_K K[y_{j}],
		\end{equation} 
		\begin{equation}\label{c2}
			K[V(K_{n}\odot H)]/(I(K_{n}\odot H),y_{j})\cong K[V(K_{n-1}\odot H)]/I(K_{n-1}\odot H)\tensor_K K[V(H)]/I(H) \tensor_K K[i(H)],
		\end{equation}
	\begin{Theorem}\label{The119}	Let $n\geq1$ and $ S=K[V(K_{n}\odot H)].$ Then
		\begin{itemize}
			\item [(a)]
			$	\depth(S/I(K_{n}\odot H))=1+(n-1)\big(\depth(K[V(H)]/I(H))+|i(H)|\big).$ \item[(b)]  $\sdepth(S/I(K_{n}\odot H))\geq 1+(n-1)\big(\sdepth(K[V(H)]/I(H))+|i(H)|\big).$

		\noindent	In particular, if 	$H$ is a null graph, then  $$	\sdepth(S/I(K_{n}\odot H))=1+(n-1)|V(H)|.$$
	
		\end{itemize}
		
	\end{Theorem}
	\begin{proof}
		%		We consider two cases. \begin{description}
		%		\item[Case 1] Let $b=0$. If $m=3$ and $n=1$ by using CoCoA \cite{cocoa}, we have $	\depth(S_{3(1+1)}/I(K_{3}\odot H))=3$. If $m=3$ and $n=2$ by using CoCoA \cite{cocoa}, we have $	\depth(S_{3(2+1)}/I(K_{3}\odot H))=5$. If $m=3$ and $n=3$ by using CoCoA \cite{cocoa}, we have $	\depth(S_{3(3+1)}/I(K_{3}\odot H))=7$. If $m=3$ and $n\geq1$, we have $	\depth(S_{3(n+1)}/I(K_{3}\odot H))=3+2(n-1).$ We can prove this by induction on $m.$ 
		%		For $w_{i}\in K_{3}$, by using  Lemmas \ref{le3} and \ref{le6} we have 	$$\depth(S_{m(n+1)}/I(K_{m}\odot H):w_{i})=1+2n,$$ 
		%		$$\depth(S_{m(n+1)}/I(K_{m}\odot H),w_{i})=n+n+1.$$ Hence, by  Lemma \ref{le41}, we have $	\depth(S_{m(n+1)}/I(K_{m}\odot H))=1+2n.$ Now for $m\geq 3$ and $n\geq1$, we will prove this by induction on $m.$ 
		%		For $w_{i}\in K_{m}$, by using  Lemma \ref{le3}  we have 	$$\depth(S_{m(n+1)}/I(K_{m}\odot H):w_{i})=1+(m-1)n,$$ 
		%		$$\depth(S_{m(n+1)}/I(K_{m}\odot H),w_{i})=n+1+(m-2)n=1+(m-1)n.$$ Hence, by  Lemma \ref{le41}, we have $	\depth(S_{m(n+1)}/I(K_{m}\odot H))=1+(m-1)n.$
		%		\item[Case 2] For $b\neq 0$, w
		%		
We first prove the result for depth.	Let $t:=\depth(K[V(H)]/I(H)).$	If $ n=1,2,$ $K_1\cong N_1$ and $K_2\cong P_2,$ we get the required result  by using Lemma \ref{triv} and Theorem \ref{The1}. %If $n=3$, Let $n=3$ and $|i(H)|\geq1$. For $y_{2} \in C_{3}$, we get
		%			$\big(K[V(\mathbb{C}_{3}/y_{2})]/I(\mathbb{C}_{3}/y_{2}))\big[\overline{Q}_{\mathbb{C}_{3}}(y_{2}) ] \cong \underset{l=1}{\overset{2}{\tensor}}K[i(H)] \tensor K[y_{2}],$
		%			so by Lemma \ref{le3} and \ref{LEMMA1.5}, we have
		%			$\depth(S/(I(\mathbb{C}_{3}):y_{2}))= 2|i(H)|+1.$ Also $K[V(\mathbb{C}_{3}\backslash y_{2})]/I({(\mathbb{C}_{3})}\backslash y_{2})[P_{\mathbb{C}_{3}}(y_{2})]\cong K[V(\mathbb{P}_{2})]/I(\mathbb{P}_{2})\tensor K[i(H)] ,$ so by using  Lemma \ref{le3} and Theorem \ref{The1}, it follows that
		%			$\depth(S/(I(\mathbb{C}_{n}),y_{n-1}))=2|i(H)|+1.$
		%			By applying Lemma \ref{le41}, we have the required result. 
		%			Let $n\geq3$ and $y_{n}\in \mathbb{K}_{n}$., consider the following short exact sequence
		%			\begin{equation}\label{161}
		%				0\longrightarrow S/(I(\mathbb{K}_{n}):y_{n})\xrightarrow{\cdot y_{n}} S/I(\mathbb{K}_{n})\longrightarrow S/(I(\mathbb{K}_{n}),y_{n})\longrightarrow 0,
		%			\end{equation}
		%			by Depth Lemma
		%			\begin{equation*}
		%				\depth(S/I(\mathbb{K}_{n}))\geq \min\{\depth(S/(I(\mathbb{K}_{n}):y_{n})), \depth(S/(I(\mathbb{K}_{n}),y_{n})\}.
		%			\end{equation*}
		%			Here
	By  applying Lemma \ref{le3}, Lemma \ref{LEMMA1.5} on Eq. \ref{c1} and Eq. \ref{c2}, we have
 \begin{equation*}
     \begin{split}
         \depth( S/(I(K_n \odot H):y_{j}))&=(n-1) \depth(K[V(H)]/I(H))+ (n-1)\depth(K[i(H)])\\&\quad\quad +\depth( K[y_{j}])\\ &=(n-1)t+(n-1)|i(H)|+1\\&= 1+(n-1)(t+|i(H)|)
     \end{split}
 \end{equation*}
	and by induction on $n,$ \begin{equation*}
	    \begin{split}
	         \depth(  S/(I(K_n \odot H),y_{j}))&=\depth(K[V(K_{n-1}\odot H)]/I(K_{n-1}\odot H))+\depth( K[V(H)]/I(H))  \\&\quad \quad+\depth(K[i(H)])\\&\quad=1+(n-2)\big(t+|i(H)|\big)+t+|i(H)|\\&\quad=1+(n-1)(t+|i(H)|).
	    \end{split}
	\end{equation*} 
		Thus the required result for the depth follows by Lemma \ref{exacther}.

		Next, to prove the result for Stanley depth, consider if $H$ is not a null graph, then we get the required inequality for Stanley depth by applying Lemma \ref{le3}, Lemma \ref{LEMMA1.5},   and Lemma \ref{rle2} on  Eq. \ref{c1} and Eq. \ref{c2}.
	If	$H$ is a null graph, then  by   Eq. \ref{c1} and Eq. \ref{c2},  
		\begin{equation*}
			S/(I(K_n \odot H):y_{j})\cong   \underset{l=1}{\overset{n-1}{\tensor_K}}K[V(H)] \tensor_K K[y_{j}],
		\end{equation*} 
		\begin{equation*}
			S/(I(K_n \odot H),y_{j})\cong K[V(K_{n-1}\odot H)]/I(K_{n-1}\odot H) \tensor_K K[V(H)], \end{equation*} and we get the desired result for Stanley depth similarly as we obtained for depth just by replacing Lemma \ref{exacther} by Lemma \ref{exacthersdepth}.
		
		%\end{description}
	\end{proof} 
		
	\begin{Corollary}  Stanley's inequality holds for $S/I(K_n \odot H)$ if it holds for  $K[V(H)]/I(H).$
	\end{Corollary}
	%	\begin{Remark}	\em{	If $H=i(H),$ then $	\depth(S/I(\mathbb{K}_{n}))=\sdepth(S/I(\mathbb{K}_{n}))= 1+(n-1)|i(H)|.$ For $n=2,$ we use  Remark \ref{b3} and  by imitating  the proof of Theorem \ref{The119} we get the required equality. }	\end{Remark}
	\begin{Corollary} \label{corsub}
		Let $n\geq 1.$  Then 	
		\begin{equation*}
			\pdim(S/I(K_n \odot H))=
			n(|V(H)|+1)-1-(n-1)\big(\depth(K[V(H)]/I(H))+|i(H)|\big).
		\end{equation*}
	\end{Corollary}
	\begin{proof}
		As $|V(K_n \odot H)|=	n(|V(H)|+1)$.	By using Lemma \ref{auss13} and Theorem \ref{The119}, one can  find the required result.
	\end{proof}
	\begin{Corollary}	Let $n,m,s\geq1$ and $ S=K[V(K_{n}\odot K_{m})].$ Then
		\begin{itemize}
			\item [(a)]
			$	\depth(S/I(K_{n}\odot K_{m}))=n.$ %\item[(b)]  $\sdepth(S/I(K_{n}\odot K_{m}))\geq n.$
			\item[(b)] $\pdim(S/I(K_n \odot K_m))=
			nm.$
			\item[(c)]$\depth(K[V(Br_s(K_n))]/I(Br_s(K_n)))=\sdepth(K[V(Br_s(K_n))]/I(Br_s(K_n)))=1+(n-1)s.$
		\end{itemize}
	\end{Corollary}
	%	\subsection{Corona product of star graph with any graph H}
	%	\section{Depth, Stanley depth and regularity of residue class ring of edge ideals of corona product of star graph and any graph $H$}
	%Let $S_{n}$ be a star graph on $n+1$ vertices, say $y_1$ be the internal vertex such that $V(S_{n})=\{y_{1},\dots , y_{n+1}\}.$ %The minimal generating set for the edge ideal of 
%	For any graph $H,$	we denote $\mathbb{S}_{n}=S _{n}\odot H.$  %$\mathbb{S}_{n+1}=S_{n+1}\odot H$ 
	%	The minimal generating set for the edge ideal of $\mathbb{S}_{n+1}=S_{n+1}\odot H$ is as follows:
	%\begin{eqnarray*}	\mathcal{H}(I(\mathbb{S}_{n+1}))=\big(\cup^{n+1}_{j=1}\{y_{j}x_{j1},y_{j}x_{j2}, \dots ,y_{j}x_{j(i(H)+|A|)}\}\cup^{n+1}_{j=1}\{x_{jl}x_{jk}\in E(H_{j})| ~ \text{iff}~ x_{l}x_{k}\in E(H) \}\\ \cup\{y_1y_{j}:2\leq j\leq n+1\}\big).	\end{eqnarray*}
\noindent % Let $V(S_n)=\{y_{1},y_{j}:2\leq j\leq n+1\}$ and edge set $E(S_n)=\{\{y_{1},y_{j}\} :2\leq j\leq n+1\}$ with root vertex $y_1.$ 
%The minimal generating set for the edge ideal of $S_n \odot H$ is as follows:\begin{multline*}	\mathcal{G}(I(S_n \odot H))=\big(\cup^{n+1}_{j=1}\{y_{j}x_{j1},y_{j}x_{j2}, \dots ,y_{j}x_{j(i(H)+|A|)}\}\cup^{n+1}_{j=1}\{x_{jl}x_{jk}\in E(H_{j})| ~ \text{iff}~ x_{l}x_{k}\in E(H) \}\\ \cup\{y_1y_{j}:2\leq j\leq n+1\}\big).\end{multline*}

\begin{Theorem}\label{The1199} Let $S=K[V(S_n \odot H)]$. Then
	\begin{itemize}
		\item[(a)] 	 $	\depth(S/I(S_n \odot H))= n+\depth(K[V(H)]/I(H))+|i(H)|.$  \item[(b)] $\sdepth(S/I(S_n \odot H))\geq n+\sdepth(K[V(H)]/I(H))+|i(H)|.$ 
		
	\noindent	In particular, if $H$ is a null graph, then  $$\sdepth(S/I(S_n \odot H))= n+|V(H)|.$$
	\end{itemize}
\end{Theorem}
\begin{proof}
Let  $t:=\depth(K[V(H)]/I(H)).$	First we will prove the depth result. If $1\leq n\leq 3$, then the result follows from Lemma \ref{triv} and Theorem \ref{The1}.		
	Let $n\geq4.$ For $y_{n+1}\in S_n \odot H$, 
	we have the following isomorphisms: 
	\begin{equation*}
	    S/(I(S_n \odot H):y_{n+1})\cong\\ \underset{l=1}{\overset{n-1}{\tensor_K}}K[V(N_1 \odot H)]/I(N_1 \odot H)\tensor_K K[i(H)\cup \{y_{n+1}\}] \tensor_K K[V(H)]/I(H).\end{equation*}
	Therefore, by Lemma \ref{le3}, Lemma \ref{LEMMA1.5} and  Lemma \ref{triv}, we have 
\begin{equation*}
    \begin{split}
         \depth( S/(I(S_n \odot H):y_{n+1}))&=(n-1)\depth(K[V(N_1 \odot H)]/I(N_1 \odot H))+ \depth( K[i(H)\cup \{y_{n+1}\}]) \\&\quad\quad+ \depth(K[V(H)]/I(H))\\&= 
	n+t+|i(H)|.
    \end{split}
\end{equation*}	
It can also be seen that $$S/(I(S_n \odot H),y_{n+1})\cong K[V(S_{n-1}\odot H)]/I(S_{n-1}\odot H)\tensor_K K[i(H)] \tensor_K K[V(H)]/I(H).$$
	Therefore, by using induction,  Lemma \ref{le3} and Lemma \ref{LEMMA1.5},  we get \begin{equation*}
	    \begin{split}
	         \depth(S/(I(S_n \odot H),y_{n+1}))&=\depth(K[V(S_{n-1}\odot H)]/I(S_{n-1}\odot H))+\depth(K[i(H)])\\&\quad\quad+ \depth(K[V(H)]/I(H))\\&= n-1+t+|i(H)|+t+|i(H)|\\&=2(t+|i(H)|)+n-1.
	    \end{split}
	\end{equation*}
	  By using Lemma \ref{rle2},  we get $	\depth(S/(I(S_n \odot H)) \geq n+t+|i(H)|.$
	%	Thus the required result follows by applying  Lemma \ref{le41}.%\end{description}
	For the other inequality, let $u:=y_2y_3\cdots y_{n+1}\notin I(S_n \odot H),$ we have the following isomorphism:
	\begin{equation}\label{c3}
		S/(I(S_n \odot H):u)\cong K[\{y_{2},y_{3}, \dots, y_{n+1}\}\cup i(H)]\tensor_K K[V(H)]/I(H). 
	\end{equation}	
	%	where $S'''=K[V(S_n \odot H)\backslash \{\underset{i=2}{\overset{n+1}{\cup}}\overline{N}_{S_n \odot H}(y_i)\cup \underset{i=2}{\overset{n+1}{\cup}}Q_{S_n \odot H}(y_{i})\}].$
	By applying Lemma  \ref{le3} and Lemma  \ref{Cor7} on Eq. \ref{c3}, we get \begin{equation*}
	    \begin{split}
	         \depth (S/I(S_n \odot H))	&\leq  \depth (S/(I(S_n \odot H):u))\\& \quad=\depth(K[\{y_{2},y_{3}, \dots, y_{n+1}\}\cup i(H)])+\depth( K[V(H)]/I(H))\\& \quad=n+|i(H)|+t.
	    \end{split}
	\end{equation*}
Hence $	\depth (S/I(S_n \odot H))=n+|i(H)|+t.$

Now, we prove the result for Stanley depth. If $H$ is not a null graph, then the required lower bound of Stanley depth is obtain in a similar way as for depth by using induction, Lemma \ref{le3}, Lemma \ref{LEMMA1.5}, Lemma \ref{triv} and Lemma  \ref{rle2}. Let $H$ be a null graph, we have a following isomorphisms:
$$S/(I(S_n \odot H),y_{n+1})\cong K[V(S_{n-1}\odot H)]/I(S_{n-1}\odot H)\tensor_K K[V(H)],$$
 \begin{equation*}
	    S/(I(S_n \odot H):y_{n+1})\cong \underset{l=1}{\overset{n-1}{\tensor_K}}K[V(N_1 \odot H)]/I(N_1 \odot H)\tensor_K K[V(H)\cup \{y_{n+1}\}],
	\end{equation*}
 
 	\begin{equation*}
		S/(I(S_n \odot H):u) \cong K[V(H)\cup \{y_{2},y_{3}, \dots, y_{n+1}\}]. 
	\end{equation*}
 The  proof for  Stanley depth is similar to depth.
 
\end{proof}%\noindent  If	Stanley's inequality holds for $K[V(H)]/I(H),$ then it also holds  for $S/I(S_n \odot H).$ 

\begin{Corollary}
	If	Stanley's inequality holds for $K[V(H)]/I(H),$ then it also holds  for $S/I(S_n \odot H).$
\end{Corollary}
\begin{Corollary} \label{corsub}
	If $n\geq 1$, then
	\begin{equation*}
		\pdim(S/I(S_n \odot H))=
		(n+1)(|V(H)|+1)-n-\depth(K[V(H)]/I(H))-|i(H)|.
	\end{equation*}
\end{Corollary}

\begin{proof}
	As $|V(S_n \odot H)|=(n+1)(|V(H)|+1)$.	By using Lemma \ref{auss13} and Theorem \ref{The1199}, the required result follows.
\end{proof}
\begin{Corollary} Let $n,m,s\geq 1$ and $ S=K[V(S_n \odot S_m)]$. Then
	\begin{itemize}
		\item[(a)] 	 $	\depth(S/I(S_n \odot S_m))= n+1.$  %\item[(b)] $\sdepth(S/I(S_n \odot S_m))\geq n+1.$   
		\item[(b)] $	\pdim(S/I(S_n \odot S_m))=
		(n+1)(m+1).$
			\item[(c)]$\depth(K[V(Br_s(S_n))]/I(Br_s(S_n)))=\sdepth(K[V(Br_s(S_n))]/I(Br_s(S_n)))=n+s.$ 
	\end{itemize}
\end{Corollary}
	\noindent %Let $K_{u,v}$ denotes the complete bipartite graph with partite sets $K_u=\{y_{1},y_{2},\dots, y_{u}\}$, $K_v=\{y_{u+1},y_{u+2},\dots, y_{u+v}\},$ and edge set $E(K_{u,v})=\{\{y_{i},y_{j}\}|~~1\leq i\leq u ~ \text{and}~u+1\leq j\leq u+ v\}.$
	%The minimal generating set for the edge ideal of $K_{u,v} \odot H$ is as follows:\begin{multline*}		\mathcal{G}(I(K_{u,v} \odot H))=\big(\cup^{u}_{i=1}\{y_{i}x_{i1},y_{i}x_{i2}, \dots ,y_{i}x_{i(|i(H)|+|A|)}\}\cup^{u+v}_{j=u+1}\{y_{j}x_{j1},y_{j}x_{j2}, \dots ,y_{j}x_{j(|i(H)|+|A|)}\}\\ \cup\{y_{i}y_{j}: 1\leq i\leq u ~ \text{and}~u+1\leq j\leq u+v \}\cup^{u+v}_{j=1}\{x_{jl}x_{jk}\in E(H_{j})| ~ \text{iff}~ x_{l}x_{k}\in E(H)\}\big).\end{multline*}
For $y_{u+v}\in V(K_{u,v} \odot H),$ we have the following isomorphism: 
\begin{multline}\label{e1}
    	K[V(K_{u,v} \odot H)]/(I(K_{u,v} \odot H):y_{u+v})\cong \\ \underset{l=u+1}{\overset{u+v-1}{\tensor_K}}K[V(N_1 \odot H)]/I(N_1 \odot H) \underset{l=1}{\overset{u}{\tensor_K}} K[i(H)]\underset{l=1}{\overset{u}{\tensor_K}} K[V(H)]/I(H)  \tensor_K K[y_{u+v}],
\end{multline}
	\begin{multline}\label{e2}
		K[V(K_{u,v} \odot H)]/(I(K_{u,v} \odot H),y_{u+v})\cong\\ K[V(K_{u,v-1}\odot H)]/I(K_{u,v-1}\odot H) \tensor_K K[i(H)] \tensor_K K[V(H)]/I(H) .
	\end{multline}
\begin{Theorem}\label{The1177}
If  $u,v \in \mathbb{Z}^+$ and $S=K[V(K_{u,v} \odot H)],$  then
	\begin{itemize}
		\item [(a)] $	\depth(S/I(K_{u,v} \odot H))=\min\{u,v\}\big(\depth(K[V(H)]/I(H))+|i(H)|\big)+\max\{u,v\}.$
		\item[(b)] $\sdepth(S/I(K_{u,v} \odot H))\geq \min\{u,v\}\big(\sdepth(K[V(H)]/I(H))+|i(H)|\big)+\max\{u,v\}.$
		
	\noindent	In particular, if $H$ is a null graph then  $$\sdepth(S/I(K_{u,v} \odot H))= \min\{u,v\}|V(H)|+\max\{u,v\}.$$
	%	Otherwise,   $\sdepth(S/I(K_{u,v} \odot H))\geq \min\{u,v\}\big(\sdepth(K[V(H)]/I(H))+|i(H)|\big)+\max\{u,v\}.$
	\end{itemize}
\end{Theorem}
\begin{proof}
	
	Let  $t:=\depth(K[V(H)]/I(H)),$ without loss of generality, we assume that $v\geq u.$  For $u=v=1,$ the result follows from Theorem \ref{The1}.	If $u=1$ and $v\geq 1,$ then the result follows from Theorem \ref{The1199}.
	Let $u,v\geq 2.$
		By using induction on $v$, Lemma  \ref{le3}, Lemma \ref{LEMMA1.5} and Lemma \ref{triv} on Eq. \ref{e1} and Eq. \ref{e2}, we have
  \begin{equation*}
      \begin{split}
           \depth(S/(I(K_{u,v} \odot H):y_{u+v}))&=(v-1)\depth(K[V(N_1 \odot H)]/I(N_1 \odot H)) +u\cdot\depth( K[i(H)])\\&\quad\quad+u\cdot \depth( K[V(H)]/I(H))  +\depth( K[y_{u+v}])\\&= v-1 +u\big(|i(H)|+t\big)+1\\&= u(|i(H)|+t)+v, 
      \end{split}
  \end{equation*}
		\begin{equation*}
		    \begin{split}
		         \depth(S/(I(K_{u,v} \odot H),y_{u+v}))&=\depth(K[V(K_{u,v-1}\odot H)]/I(K_{u,v-1}\odot H))+\depth( K[i(H)]) \\&\quad\quad+\depth( K[V(H)]/I(H))\\& =u(t+|i(H)|)+v-1+t+|i(H)|\\&=(u+1)\big(|i(H)|+t\big)+v-1.
		    \end{split}
		\end{equation*} Since $(u+1)\big(|i(H)|+t\big)+v-1 \geq u(|i(H)|+t)+v.$ By Lemma \ref{rle2}, we get $ \depth(S/I(K_{u,v} \odot H)) \geq u(|i(H)|+t)+v.$	 For the other inequality, let $w:=y_{u+1}y_{u+2}\cdots y_{u+v}\notin I(K_{u,v} \odot H),$  %let $u:=y_{u+1}y_{u+2}\cdots y_{u+v}\notin I(K_{u,v} \odot H),$
	we have the following $K$-algebra isomorphism:
	\begin{equation}\label{d3}
		S/(I(K_{u,v} \odot H):w) \cong\\ K[y_{u+1},y_{u+2}, \dots, y_{u+v}]\underset{l=1}{\overset{u}{\tensor_K}} K[i(H)]\underset{l=1}{\overset{u}{\tensor_K}} K[V(H)]/I(H).
	\end{equation}
	%	where $S'''=K[V(K_{u,v} \odot H)\backslash \{\underset{i=u+1}{\overset{u+v}{\cup}}\overline{N}_{K_{u,v} \odot H}(y_i)\cup \underset{i=u+1}{\overset{u+v}{\cup}}Q_{K_{u,v} \odot H}(y_{i})\}].$
	By applying Lemma  \ref{le3}, Lemma  \ref{LEMMA1.5} and Lemma  \ref{Cor7} on Eq. \ref{d3}, we have
	\begin{equation*}
 \begin{split}
	    	\depth (S/I(K_{u,v} \odot H))	&\leq  \depth (S/(I(K_{u,v} \odot H):w))\\&=\depth(K[y_{u+1},y_{u+2}, \dots, y_{u+v}])+u\cdot \depth( K[i(H)])\\&\quad\quad+u\cdot \depth(K[V(H)]/I(H))\\&=u(|i(H)|+t)+v.
 \end{split}
	\end{equation*}
	Hence $	\depth (S/I(K_{u,v} \odot H))=u(|i(H)|+t)+v.$ 
Now, we prove the result for Stanley depth.  If $H$ is not a null graph, then we get the required inequality for Stanley depth in a similar way to depth by using induction on $v$, Lemma \ref{le3}, Lemma \ref{LEMMA1.5}, Lemma \ref{triv} and Lemma \ref{rle2} on Eq. \ref{e1} and Eq. \ref{e2}.	But if $H$ is a null graph, then we have the following isomorphisms:
	\begin{equation*}
		S/(I(K_{u,v} \odot H),y_{u+v})\cong K[V(K_{u,v-1}\odot H)]/I(K_{u,v-1}\odot H) \tensor_K K[V(H)],
	\end{equation*}
\begin{equation*}
    	S/(I(K_{u,v} \odot H):y_{u+v})\cong  \underset{l=u+1}{\overset{u+v-1}{\tensor_K}}K[V(N_1 \odot H)]/I(N_1 \odot H) \underset{l=1}{\overset{u}{\tensor_K}} K[V(H)] \tensor_K K[y_{n+1}],
\end{equation*}

\begin{equation*}
		S/(I(K_{u,v} \odot H):w) \cong K[y_{u+1},y_{u+2}, \dots, y_{u+v}]\underset{l=1}{\overset{u}{\tensor_K}} K[V(H)],
	\end{equation*}
and the proof for Stanley depth is similar to the depth. 
\end{proof}%\noindent  Stanley's inequality holds  for $S/I(K_{u,v} \odot H)$ if it holds for $K[V(H)]/I(H).$

\begin{Corollary}
	If	Stanley's inequality holds for $K[V(H)]/I(H),$ then it also holds  for $S/I(K_{u,v} \odot H).$
\end{Corollary}
\begin{Corollary} \label{corsub}
	If $S=K[V(K_{u,v} \odot H)]$, then
	\begin{equation*}
		\pdim(S/I(K_{u,v} \odot H))=
		(u+v)(|V(H)|+1) -\min\{u,v\}\big(\depth(K[V(H)]/I(H))+|i(H)|\big)-\max\{u,v\}.
	\end{equation*}
\end{Corollary}
\begin{proof}
	As $|V(K_{u,v} \odot H)|=(u+v)(|V(H)|+1)$.	By using Lemma \ref{auss13} and Theorem \ref{The1177}, we get the required result.
\end{proof}
\begin{Lemma}[{\cite[Lemma 1.1]{POP}}]
\em{%For a complete bipartite graph $K_{u,v},$ 
If $S=K[V(K_{u,v})],$  then  $\depth(S/I({K_{u,v}}))=1.$}
\end{Lemma}
\begin{Corollary}
	Let $u,v,m,n,s\in \mathbb{Z}^+$  
and $S=K[V(K_{u,v} \odot K_{m,n})]$. Then
	\begin{itemize}
		\item [(a)] $	\depth(S/I(K_{u,v} \odot K_{m,n}))=u+v.$
	%	\item[(b)]   $\sdepth(S/I(K_{u,v} \odot K_{m,n}))\geq u+v.$
		\item[(b)] $	\pdim(S/I(K_{u,v} \odot K_{m,n}))=(u+v)(m+n).$
			\item[(c)]$\depth(K[V(Br_s(K_{u,v}))]/I(Br_s(K_{u,v})))=\sdepth(K[V(Br_s(K_{u,v}))]/I(Br_s(K_{u,v})))=\\\min\{u,v\}s+\max\{u,v\}.$ 
	\end{itemize}
\end{Corollary}
%	\begin{Theorem}\label{The118}
%		Let $H$ be any graph on $q$ vertices in which $a$ vertices are isolated and $b$ are non-isolated with $\sdepth(S_{b}/I(H^{\star}_{c}))=h_1$, and $\sdepth(S_{q}/I(H))=h_1+a=h\geq 1$. For $u,v\geq 1$ and $v\geq u$ then
%		\begin{eqnarray*}
%			\sdepth(S_{(u+v)(q+1)}/I(\mathcal{K}_{u,v}))=uh+v.
%		\end{eqnarray*}
%	\end{Theorem}

		\subsection{Regularity}\label{sec2.2}
\paragraph{} \bigskip  We introduce a lemma at the beginning of this subsection,  which will be used very often. It is obvious that $\reg(S)=0.$
		
	\begin{Lemma}\label{regtriv}\bigskip
If $ S=K[V(N_1 \odot H)],$  then
 $$
\reg(S/I(N_1 \odot H))=
\begin{cases}
	1, & \text{if}\, $H$\, \text{is a null graph;}\\  
 \reg(K[V(H)]/I(H)), &  \text{otherwise.}
\end{cases}
$$
	\end{Lemma}
	\begin{proof}
		If $H$ is a null graph, %and $i(H)\geq 1$, 
		then $|V(H)|=|i(H)|$ and $N_1 \odot H\cong S_{|V(H)|}$. By Lemma \ref{reg1}, we have $\reg(S/I(N_1 \odot H))=1.$ If $H$  is not a null graph, %As$\big(K[V(\mathbb{K}_{1}/y_{1})]/I(\mathbb{K}_{1}/y_{1}))\big[\overline{Q}_{\mathbb{K}_{1}}(y_1)]\cong K[y_{1}].$ Also we have $K[V(\mathbb{K}_{1}\backslash y_1)]/I(\mathbb{K}_{1}\backslash y_1)[P_{\mathbb{K}_{1}}(y_1)]\cong  K[V(H')]/I(H')\tensor K[i(H)].$
		then by Eq. \ref{vertexcolon}, we get $\reg(S/(I(N_1 \odot H):y_{1}))=\reg(K[y_1])=0.$  By using	Lemma \ref{le3}  in  Eq.  \ref{vertexcoma}, we have $\reg(S/(I(N_1 \odot H),y_{1}))=\reg(K[V(H)]/I(H))\geq 1$. By  Lemma \ref{regul2}(c), we get $ 	\reg(S/I(N_1 \odot H))=\reg(K[V(H)]/I(H)).$
	\end{proof}
		%	\section{Krull dimension of residue class ring of edge ideals of corona product of two graphs}
		
		\begin{Remark}\label{zeroandtriv}
\em{While proving the result by induction on $n$, we have a module of the type $K[V(P_0\odot H)]/I(P_0\odot H),$	so in that case  we define $\reg(K[V(P_0\odot H)]/I(P_0\odot H))=0.$ }
	\end{Remark}
	\begin{Theorem}\label{THE2}
	Let $n\geq 1.$ If $S=K[V(P_n \odot H)],$ then
	$$
	\reg(S/I(P_n \odot H))=
	\begin{cases}
		\big\lceil \frac{n}{2}\big\rceil, & \text{if}\, $H$\, \text{is a null graph;}\\  
		n\cdot \reg(K[V(H)]/I(H)), & \text{otherwise.}
	\end{cases}
	$$
	%			If $r=0$,  then
	%			
	%			\begin{equation*}
	%				\reg(S/I(P_n \odot H)=\bigg\lceil \frac{n}{2}\bigg\rceil.
	%=\left\{\begin{matrix}
	%			\frac{n}{2}, & if \quad \quad \text{$n$ is even} ;\\ \\ \frac{n+1}{2} & if \quad \quad\text{ $n$ is odd},
	%		\end{matrix}\right.\end{equation*}
	%			and for $r\geq 1$, we have
	%			\begin{equation*}
	%				\reg(S_{n(q+1)}/I(P_{n}\odot H))=nr.\end{equation*}
\end{Theorem}
\begin{proof} We consider two cases. 
	\begin{description}
		\item[Case 1] Let $H$ is a null graph.  
		If $n=1,$ result follows by Lemma \ref{regtriv}. 	If $n=2,$ we have $S/(I(P_{2}\odot H):y_{1})\cong K[V(H)]\tensor_{K}K[y_1]$ and $S/(I(P_{2}\odot H),y_1)\cong K[V(N_{1}\odot H)]/I(N_{1}\odot H)\tensor_{K} K[V(H)].$ We get $\reg(S/(I(P_{2}\odot H):y_{1}))= 0$ and by Lemma \ref{le3}, Lemma \ref{regtriv},  we have $\reg(S/(I(P_{2}\odot H),y_{1}))=\reg( K[V(N_{1}\odot H)]/I(N_{1}\odot H))+0=1.$ Thus by using  Lemma \ref{regul2}(c), we get $	\reg(S/I(P_{2}\odot H)) =1.$ 	Let $n\geq 3.$	 	For $y_{n-1}\in V(P_n \odot H),$ we have the following $K$-algebra isomorphisms:   $$S/(I(P_n \odot H):y_{n-1})\cong K[V(P_{n-3}\odot H)]/I(P_{n-3}\odot H)\underset{l=1}{\overset{2}{\tensor_{K}}}K[V(H)]\tensor_{K}K[y_{n-1}] ,$$
		\begin{equation*}
  \begin{split}
      &S/(I(P_n \odot H),y_{n-1})\cong \\&  K[V(P_{n-2}\odot H)]/I(P_{n-2}\odot H)\tensor_{K} K[V(N_{1}\odot H)]/I(N_{1}\odot H) \tensor_{K} K[V(H)]. 
  \end{split}\end{equation*}
	%	For $y_{n}\in V(P_n \odot H),$ we have  $S''/I((P_n \odot H)/y_{n})\big  [\overline{Q}_{P_n \odot H}(y_n)]\cong K[V(P_{n-2}\odot H)]/I(P_{n-2}\odot H)\tensor_{K} K[V(H),y_{n}]$ 	and	$S'/I((P_n \odot H)\backslash y_n)[P_{P_n \odot H}(y_n)]\cong K[V(P_{n-1}\odot H)]/I(P_{n-1}\odot H)\tensor_{K} K[V(H)].$
	%	If $n=2,$   by Remark \ref{zeroandtriv}, $\reg(S''/I((P_{2}\odot H)/y_{2})\big  [\overline{Q}_{P_{2}\odot H}(y_2)])= 0$ and by Lemma \ref{le3}, Lemma \ref{regtriv},  we get $\reg(S'/I((P_{2}\odot H)\backslash y_2)[P_{P_{2}\odot H}(y_2)])=1.$ Thus by using  Remark \ref{regul}(c), we get $	\reg(S/I(P_{2}\odot H)) =1.$
	%	\begin{multline*}    \reg(S''/I((P_n \odot H)/y_{n})\big  [\overline{Q}_{P_n \odot H}(y_n)])\\=\reg( K[V(P_{n-2}\odot H)]/I(P_{n-2}\odot H))+ \reg(K[V(H),y_{n}])=\big\lceil \frac{n-2}{2}\big\rceil =\big\lceil\frac{n}{2}\big\rceil-1,	\end{multline*}		$$\reg(S'/I((P_n \odot H)\backslash y_n)[P_{P_n \odot H}(y_n)])=\reg(K[V(P_{n-1}\odot H)]/I(P_{n-1}\odot H))+ \reg(K[V(H)])=\big\lceil\frac{n-1}{2}\big\rceil.$$
	%	If $n$ is even, we have $\big\lceil\frac{n-1}{2}\big\rceil=\big\lceil\frac{n}{2}\big\rceil.$	By using  Remark \ref{regul}(c), we have $	\reg(S/I(P_n \odot H))=\big\lceil\frac{n}{2}\big\rceil.$ If $n$ is odd, we have  $\big\lceil\frac{n-1}{2}\big\rceil=\big\lceil\frac{n}{2}\big\rceil-1.$ 
	%	Thus by using  Remark \ref{regul}(b), we get $	\reg(S/I(P_n \odot H)) \in \{ \big\lceil\frac{n}{2}\big\rceil, \big\lceil\frac{n-1}{2}\big\rceil-1\}.$ 
Using induction on $n$ and Lemma \ref{le3}, we have \begin{equation*}
		    \reg(S/(I(P_n \odot H):y_{n-1}))=\reg(K[V(P_{n-3}\odot H)]/I(P_{n-3}\odot H))=\lceil \frac{n-3}{2} \rceil.
		\end{equation*}  By induction on, Lemma \ref{le3}, Lemma \ref{circulentt}
		and Lemma  \ref{regtriv}, we get \begin{equation*}\begin{split}
		   & \reg(S/(I(P_n \odot H),y_{n-1}))\\&=\reg(K[V(P_{n-2}\odot H)]/I(P_{n-2}\odot H))+ \reg(K[V(N_{1}\odot H)]/I(N_{1}\odot H))\\&=\lceil\frac{n}{2}\rceil.
		\end{split}\end{equation*} Since $\lceil \frac{n-3}{2} \rceil< \lceil\frac{n}{2}\rceil,$ thus by using  Lemma \ref{regul2}(c), we get $	\reg(S/I(P_n \odot H))=\lceil\frac{n}{2}\rceil.$
		
		\item[Case 2] Let $H$ is not a null graph and $\reg(K[V(H)]/I(H))=r$. % Let $n=1.$ We have $P_{1}=\{y_{1}\}$ and	$\big(K[V(\mathbb{P}_{1}/y_{1})]/I(\mathbb{P}_{1}/y_{1}))\big[\overline{Q}_{\mathbb{P}_{1}}(y_1)]\cong K[y_{1}].$ Also we have $K[V(\mathbb{P}_{1}\backslash y_1)]/I(\mathbb{P}_{1}\backslash y_1)[P_{\mathbb{P}_{1}}(y_1)]\cong  K[V(H')]/I(H')\tensor K[i(H)].$	 We have $\reg(S/(I(\mathbb{P}_{1}):y_{1}))=0$ and $\reg(S/(I(\mathbb{P}_{1}),y_{1}))=r$. Hence by using Theorem \ref{regul}(c), we have $	\reg(S/I(\mathbb{P}_{1}))=r.$ 
		If $n=1,$  we get the required result  by using Lemma \ref{regtriv}. %	If $n= 2,3,$ applying  Lemmas \ref{le3}, \ref{circulentt}	and \ref{regtriv} on Eqs. \ref{path1} and \ref{path2}, we have  $\reg(S''/I(\mathbb{P}_{2}/y_{2})[\overline{Q}_{\mathbb{P}_{2}}(y_2)])=r,$  $\reg(S'/I(\mathbb{P}_{2}\backslash y_2)[P_{\mathbb{P}_{2}}(y_{2})])=2r,$ $\reg(S''/I(\mathbb{P}_{3}/y_{3})[\overline{Q}_{\mathbb{P}_{3}}(y_3)])=2r$ and  $\reg(S'/I(\mathbb{P}_{3}\backslash y_3)[P_{\mathbb{P}_{3}}(y_{3})])=3r.$ Therefore, by  Remark \ref{regul}(c), $	\reg(S/I(\mathbb{P}_{2}))=2r$ and $	\reg(S/I(\mathbb{P}_{3}))=3r.$ 
		Let  $n\geq 2.$	%As we know that
		%$\big(K[V(P_n \odot H/y_{n})]/I(P_n \odot H/y_{n}))\big[\overline{Q}_{P_n \odot H}(y_n)]\cong K[V(\mathbb{P}_{n-2})]/(I(\mathbb{P}_{n-2})\tensor K[V(H')]/I(H')\tensor K[i(H),y_{n}]$ and $K[V(P_n \odot H\backslash y_n)]/I(P_n \odot H\backslash y_n)[P_{P_n \odot H}(y_n)]\cong K[V(\mathbb{P}_{n-1})]/I(\mathbb{P}_{n-1})\tensor K[V(H')]/I(H')\tensor K[i(H)].$ 
		Considering Eq. \ref{path1} and Eq. \ref{path2} and applying  Lemma \ref{le3}, Lemma \ref{circulentt} and using induction on $n$, we have	\begin{equation*}\begin{split}
		    \reg(S/(I(P_n \odot H):y_{n}))&=
		    \reg(K[V(P_{n-2} \odot H)]/I(P_{n-2} \odot H))+ \reg(K[V(H)]/I(H))\\&
		     =(n-2)r+r\\&=(n-1)r,\end{split}\end{equation*} 
		\begin{equation*}\begin{split}  \reg(S/(I(P_n \odot H), y_{n}))&=\reg(K[V(P_{n-1} \odot H)]/I(P_{n-1}\odot H))+ \reg(K[V(H)]/I(H))\\&=(n-1)r+r\\&=nr.
		\end{split}\end{equation*} Hence by using  Lemma \ref{regul2}(c), we get $	\reg(S/I(P_n \odot H))=nr.$ This completes the proof.
		
		% If $\reg(S_{n(q+1)}/I(P_{n}\odot H))=  \frac{n-1}{2},$ this will lead us to the contradiction. Since, $n$ is odd and for $n=3$, we have $\reg(S_{n(q+1)}/I(P_{n}\odot H))=1$ by using given statement. But graph is a chordal graph and we have more than one induced matching. Therefore, 
	\end{description}
\end{proof}
\begin{Corollary}
	Let $n,m\geq 1$ and $q\geq3.$ Then
	\begin{itemize}
	    \item[(a)]  $	\reg(K[V(P_n \odot P_m)]/I(P_n \odot P_m))= n\cdot 	\big\lceil \frac{m-1}{3}\big\rceil.$
	    \item[(b)]  $	\reg(K[V(P_n \odot C_q)]/I(P_n \odot C_q))= n\cdot \big\lfloor \frac{q+1}{3}\big\rfloor.$ 
	    \item[(c)]  $	\reg(K[V(P_n \odot S_m)]/I(P_n \odot S_m))= n.$ 
	\end{itemize}
\end{Corollary}

	\begin{Theorem}\label{regcn}
	Let	$n\geq 3$. If $ S=K[V(C_n \odot H)]$, then
	$$
	\reg(S/I(C_n \odot H))=
	\begin{cases}
		\big\lceil \frac{n-1}{2}\big\rceil,  & \text{if}\, $H$\, \text{is a null graph;}\\  
		n\cdot \reg(K[V(H)]/I(H)), & \text{otherwise.}
	\end{cases}
	$$
	
\end{Theorem}
\begin{proof} 
	We consider two cases.
	\begin{description}
		\item[Case 1] Let  $H$ is a null graph.
		For $y_{n}\in V(C_n \odot H)$, it is easy to see that $$S/(I(C_n \odot H):y_{n}) \cong K[V(P_{n-3}\odot H)]/I(P_{n-3}\odot H)\underset{l=1}{\overset{2}{\tensor_K}}K[V(H)] \tensor_{K} K[y_{n}],$$ 	    $$S/(I(C_n \odot H),y_{n})\cong K[V(P_{n-1}\odot H)]/I(P_{n-1}\odot H)\tensor_{K} K[V(H)].$$ By using  Lemma \ref{le3} and Theorem \ref{THE2}, we have
		\begin{equation*}
		    \reg(S/(I(C_n \odot H):y_{n}) )=\reg(K[V(P_{n-3}\odot H)]/I(P_{n-3}\odot H))=\big\lceil \frac{n-3}{2}\big\rceil,
		\end{equation*}
		\begin{equation*}
		    \reg(S/(I(C_n \odot H),y_{n}) )=\reg(K[V(P_{n-1}\odot H)]/I(P_{n-1}\odot H))=\big\lceil \frac{n-1}{2}\big\rceil.
		\end{equation*} Since $\big\lceil\frac{n-3}{2}\big\rceil< \big\lceil\frac{n-1}{2}\big\rceil.$	Hence by using  Lemma \ref{regul2}(c), we have $	\reg(S/I(C_n \odot H))=\lceil \frac{n-1}{2}\rceil.$ 
		
		\item[Case 2] Let $H$ is not a null graph and $\reg(K[V(H)]/I(H))=r$. By using  Theorem \ref{THE2},  Lemma \ref{le3} and Lemma \ref{circulentt} on Eq. \ref{b1} and Eq. \ref{b2}, we have
		\begin{equation*}\begin{split}
		    \reg(S/I((C_n \odot H):y_{n-1}))&=\reg(K[V(P_{n-3} \odot H)]/I(P_{n-3} \odot H)) +2\cdot\reg( K[V(H)]/I(H)) \\&=(n-3)r+2r\\&=(n-1)r,\end{split}\end{equation*}
		\begin{equation*}\begin{split}
		    \reg(S/(I(C_n \odot H), y_{n-1}))&=\reg(K[V(P_{n-1}\odot H)]/I(P_{n-1}\odot H))+ \reg(K[V(H)]/I(H))\\& =(n-1)r+r\\&=nr.\end{split}\end{equation*}
	 The required result follows by using  Lemma \ref{regul2}(c), thus we have $	\reg(S/I(C_n \odot H))=nr.$ 
	\end{description}
\end{proof}
		\begin{Corollary}
	If	$n,q\geq 3$ and $m\geq1,$ then
	\begin{itemize}
	    \item[(a)]  $
	\reg(K[V(C_n \odot C_q)]/I(C_n \odot C_q))=	n\cdot 	\big\lfloor \frac{q+1}{3}\big\rfloor.
	$
	     \item[(b)]  $
	\reg(K[V(C_n \odot P_m)]/I(C_n \odot P_m))=	n\cdot 		\big\lceil \frac{m-1}{3}\big\rceil.
	$
	      \item[(c)]  $
	\reg(K[V(C_n \odot K_m)]/I(C_n \odot K_m))=	n.
	$
	\end{itemize}
\end{Corollary}
	\begin{Theorem}\label{regkn} 
	Let $n\geq 1.$ If $ S=K[V(K_n \odot H)],$  then
	$$
	\reg(S/I(K_n \odot H))=
	\begin{cases}
		1,  & \text{if}\, $H$\, \text{is a null graph;}\\  
		n\cdot \reg(K[V(H)]/I(H)), & \text{otherwise.}
	\end{cases}
	$$
%	Moreover, if $H=V(H)$,  then	$I(K_n \odot H)$ has a linear resolution.
\end{Theorem}
\begin{proof}
	For $n=1,2,$ the result follows from Lemma \ref{regtriv} and Theorem \ref{THE2}.	Let $n\geq 3$. We consider two cases here.
	\begin{description}
		\item[Case 1] Let $H$ is a null graph. In this case, consider Eq. \ref{c1} and Eq. \ref{c2} and applying induction on $n$ and Lemma \ref{le3},  we have $$\reg( S/(I(K_n \odot H):y_{j}))=0,$$ \begin{equation*}
		    \reg (S/(I(K_n \odot H), y_j))=\reg(K[V(K_{n-1}\odot H)]/(I(K_{n-1}\odot H))=1.
		\end{equation*} 
		
		Hence by using  Lemma \ref{regul2}(c), we have $	\reg(S/I(K_n \odot H))=1.$ %Also we have $(K_n \odot H)^{c}$ is a chordal graph and $\reg(I(K_n \odot H))=2$, thus by using Lemma \ref{fro}, linear resolution of $I(K_n \odot H)$ exists.
		
		\item[Case 2] If $H$ is not a null graph and $\reg(K[V(H)]/I(H))=r$. By using induction,  Lemma \ref{le3} and Lemma \ref{circulentt} on Eq. \ref{c1} and Eq. \ref{c2}, we have \begin{equation*}
		    \reg(S/(I(K_n \odot H):y_{j}))=(n-1) \reg(K[V(H)]/I(H))=(n-1)r,
		\end{equation*}	\begin{equation*}\begin{split}
		   \reg(S/(I(K_n \odot H),y_j))&=\reg(K[V(K_{n-1}\odot H)]/I(K_{n-1}\odot H))+\reg( K[V(H)]/I(H))\\&=(n-1)r+r\\&=nr.    
		\end{split}\end{equation*} 
	 Hence by using  Lemma \ref{regul2}(c), we have $	\reg(S/I(K_n \odot H))=nr.$
		
	\end{description}
\end{proof}
	\begin{Corollary} 
	Let $n,m\geq 1$ and $q\geq3.$ Then
	\begin{itemize}
	    \item [(a)]  $
	\reg(K[V(K_n \odot K_m)]/I(K_n \odot K_m))=n.$
	 \item [(b)] $
	\reg(K[V(K_n \odot P_m)]/I(K_n \odot P_m))=n\cdot \big\lceil \frac{m-1}{3}\big\rceil.$
	 \item [(c)]  $
	\reg(K[V(K_n \odot C_q)]/I(K_n \odot C_q))=n\cdot \big\lfloor \frac{q+1}{3}\big\rfloor.$
	\end{itemize}
\end{Corollary}
\begin{Theorem}\label{starreg}
	If $ S=K[V(S_n \odot H)],$	then
	$$
	\reg(S/I(S_n \odot H))=
	\begin{cases}
		n,  & \text{if}\, $H$\, \text{is a null graph;}\\  
		(n+1)\cdot \reg(K[V(H)]/I(H)), & \text{otherwise.}
	\end{cases}
	$$
\end{Theorem}
\begin{proof}
	For $n=1,2,$ the result follows from Lemma \ref{regtriv} and Theorem \ref{THE2}. Let $y_{1}$ be a vertex of degree $n$ in $S_n,$ we have the following isomorphisms: 
	\begin{equation}\label{C5}
		S/(I(S_n \odot H):y_{1})\cong \underset{l=1}{\overset{n}{\tensor_K}}K[V(H)]/I(H) \underset{l=1}{\overset{n}{\tensor_K}}K[i(H)] \tensor_K K[y_{1}],
	\end{equation} \begin{equation}\label{C4}
		S/(I(S_n \odot H), y_{1})\cong \underset{l=1}{\overset{n}{\tensor_K}}K[V(N_{1}\odot H)]/I(N_{1}\odot H)\tensor_K K[i(H)] \tensor_K K[V(H)]/I(H).
	\end{equation}  We consider two cases here.
	\begin{description}
		\item[Case 1] If $H$ is a null graph, then by using  Lemma \ref{le3}, Lemma \ref{circulentt}
		and Lemma \ref{regtriv} on Eq. \ref{C5} and Eq. \ref{C4}, we have \begin{equation*}
		    \reg(S/(I(S_n \odot H), y_{1}))=n\cdot
	\reg(K[V(N_{1}\odot H)]/I(N_{1}\odot H))=n,
		\end{equation*}
	$$\reg(S/(I(S_n \odot H):y_{1}))=0.$$ 
		Hence by using  Lemma \ref{regul2}(c), we have $	\reg(S/I(S_n \odot H))=n.$ 
		
		\item[Case 2] If $H$ is not a null graph and $\reg(K[V(H)]/I(H))=r$. Applying Lemma \ref{le3}, Lemma \ref{circulentt}
		and Lemma \ref{regtriv} on Eq. \ref{C5} and Eq. \ref{C4}, we get	$$\reg(S/(I(S_n \odot H):y_{1}))=n\cdot \reg(K[V(H)]/I(H)) =nr,$$ 
		\begin{equation*}
		    \reg(S/(I(S_n \odot H),y_{1}))=n\cdot \reg(K[V(N_{1}\odot H)]/I(N_{1}\odot H)) + \reg(K[V(H)]/I(H))=nr+r.
		\end{equation*}
	By  Lemma \ref{regul2}(c), we get $	\reg(S/I(S_n \odot H))=(n+1)r.$ \end{description}
\end{proof}
\begin{Corollary}
If $n,m\geq 1$ and $q\geq3,$ then
\begin{itemize}
    \item[(a)]  $ \reg(K[V(S_n \odot S_m)]/I(S_n \odot S_m))=	n+1.$ 
     \item[(b)]  $ \reg(K[V(S_n \odot P_m)]/I(S_n \odot P_m))=	(n+1)\cdot \big\lceil \frac{m-1}{3}\big\rceil.$ 
      \item[(c)]  $ \reg(K[V(S_n \odot C_q)]/I(S_n \odot C_q))=	(n+1)\cdot \big\lfloor \frac{q+1}{3}\big\rfloor.$ 
\end{itemize}

\end{Corollary}
\begin{Theorem}\label{combi}
Let $u,v\in \mathbb{Z^{+}}$ %suppose $v\geq u$ 
and $ S=K[V(K_{u,v} \odot H)]$. Then
	$$
	\reg(S/I(K_{u,v} \odot H))=
	\begin{cases}
		\max \{u,v\}, & \text{if}\, $H$\, \text{is a null graph;}\\  
		(u+v)\cdot \reg(K[V(H)]/I(H)), & \text{otherwise.}
	\end{cases}
	$$
\end{Theorem}
\begin{proof} 
Without loss of generality, we assume that $u\geq v.$ If $u=v=1,$ then the result follows from Theorem \ref{THE2}.	For $u\geq 1$ and $v=1,$ the  result follows from Theorem \ref{starreg}. Let $u,v\geq 2.$ Here we consider two cases.
	\begin{description}
		\item[Case 1] 	Let $H$ be a null graph. % By using induction on $u+v$, Lemma \ref{le3}, Lemma \ref{circulentt} 	and  Lemma \ref{regtriv} on Eq. \ref{f1} and Eq. \ref{f2},	\begin{multline*}\reg(S''/I((K_{u,v} \odot H)/y_{1})[\overline{Q}_{K_{u,v} \odot H}(y_1)])=\\ (u-1)\reg(K[V(N_{1}\odot H)]/I(N_{1}\odot H)) +v\cdot\reg( K[V(H)]) +\reg( K[y_{1}])= u-1,	\end{multline*}	\begin{multline*}   \reg(S'/I((K_{u,v} \odot H)\backslash y_1)[P_{K_{u,v} \odot H}(y_1)])\\=\reg(K[V(K_{u-1,v}\odot H)]/I(K_{u-1,v}\odot H)) +\reg (K[V(H)])= \max \{u-1,v\}.	\end{multline*}	If $u=v$ or $v>u,$ then by using Lemma \ref{regul}(c), we get $\reg((S/I(K_{u,v} \odot H))=v.$ But if $u>v$, we have $u-1\geq v.$ Then $$\reg(S''/I((K_{u,v} \odot H)/y_{1})[\overline{Q}_{K_{u,v} \odot H}(y_1)])=\reg(S'/I((K_{u,v} \odot H)\backslash y_1)[P_{K_{u,v} \odot H}(y_1)])=u-1.$$ By   Remark \ref{regul}(b), we get	$	\reg(S/I(K_{u,v} \odot H))\in \{u,u-1\}.	$
		Let $y_{u+v}\in V(K_{u,v} \odot H). $ 	By using induction on $v$,  Lemma \ref{le3},  Lemma \ref{circulentt} 	and  Lemma \ref{regtriv} on Eq. \ref{e1} and Eq. \ref{e2}, 
		\begin{equation*}
		    \reg(S/(I(K_{u,v} \odot H):y_{u+v}))=(v-1)\reg(K[V(N_1 \odot H)]/I(N_1 \odot H))=v-1,
	\end{equation*}\begin{equation*}
		    \reg( S/(I(K_{u,v} \odot H),y_{u+v}))=\reg(K[V(K_{u,v-1}\odot H)]/I(K_{u,v-1}\odot H)) =u.
		\end{equation*}
	 %Now in Eq.(\ref{11}), $v-1\notin \{u,u-1\}$, so	$	\reg(S/I(K_{u,v} \odot H)\neq v-1.$ 
		As $u>v,$ therefore $u> v-1$. Thus using  Lemma \ref{regul2}(c), we get
		$\reg(S/I(K_{u,v} \odot H))=u.$ \item[Case 2] Let $H$ is not a null graph and $\reg(K[V(H)]/I(H))=r$. For $y_{1}\in V(K_{u,v} \odot H) $,  we have the following $K$-algebra isomorphisms: 
	\begin{multline*}\label{f1}
		S/(I(K_{u,v} \odot H):y_{1})\cong\\ \underset{l=1}{\overset{u-1}{\tensor_K}}K[V(N_{1}\odot H)]/I(N_{1}\odot H) \underset{l=u+1}{\overset{u+v}{\tensor_K}} K[i(H)]\underset{l=u+1}{\overset{u+v}{\tensor_K}} K[V(H)]/I(H)  \tensor_K K[y_{1}],\end{multline*} 	  \begin{equation*}\label{f2}
		S/(I(K_{u,v} \odot H), y_1)\cong K[V(K_{u-1,v}\odot H)]/I(K_{u-1,v}\odot H) \tensor_K K[i(H)] \tensor_K K[V(H)]/I(H).
	\end{equation*} By using induction on $u$,  Lemma \ref{le3},  Lemma \ref{circulentt}
		and  Lemma \ref{regtriv}, we have
		\begin{equation*}\begin{split}
		    \reg(S/(I(K_{u,v} \odot H):y_{1}))&= (u-1)\reg(K[V(N_{1}\odot H)]/I(N_{1}\odot H))+v\cdot\reg( K[V(H)]/I(H))\\&=(u-1)r+vr\\&=(u+v-1)r,
		\end{split}\end{equation*}
		\begin{equation*}\begin{split}
		    \reg(S/(I(K_{u,v} \odot H),y_1))&= \reg(K[V(K_{u-1,v}\odot H)]/I(K_{u-1,v}\odot H))  +\reg( K[V(H)]/I(H))\\&=(u+v-1)r+r\\&=(u+v)r.
		\end{split}\end{equation*}
	 Hence by using  Lemma \ref{regul2}(c), we have $	\reg(S/I(K_{u,v} \odot H))=(u+v)r. $ This completes the proof.
\end{description}	\end{proof}
\begin{Corollary}
	Let $u,v,m,n\in \mathbb{Z^{+}}$ and $q\geq 3.$ Then
	\begin{itemize}
	    \item [(a)] 
	 $	\reg(K[V(K_{u,v} \odot K_{m,n})]/I(K_{u,v} \odot  K_{m,n}))=	u+v. $
	  \item [(b)] 
	 $	\reg(K[V(K_{u,v} \odot P_m)]/I(K_{u,v} \odot  P_m))=	(u+v)\cdot \big\lceil \frac{m-1}{3}\big\rceil. $
	  \item [(c)] 
	 $	\reg(K[V(K_{u,v} \odot C_q)]/I(K_{u,v} \odot  C_q))=	(u+v)\cdot \big\lfloor \frac{q+1}{3}\big\rfloor. $
	\end{itemize}
\end{Corollary}
		\section{Krull dimension and Cohen-Macaulay graphs}\label{sec3}
	In this section, we find an  expression for the Krull dimension of residue class rings of edge ideals associated with the corona product of two graphs if the Krull dimension of one graph is given. Moreover, by using the values of depth given in Section \ref{sec2}, we characterize some Cohen-Macaulay graphs.
	\begin{Theorem}\label{krulldim}
			Let $X$ and  $H$ be any two  graphs and $S=K[V(X\odot H)].$ Then
			
			\begin{equation*}
				\dim(S/I(X\odot H))=|V(X)|\cdot \big(\dim(K[V(H)]/I( H))+|i(H)|\big).\end{equation*}
		\end{Theorem}
		\begin{proof}
	%By definition \ref{defcorona}, $X_1\cong \dots\cong X_{|V(H)|} \cong X$ and $|W_1|=\dots =|W_{|V(H)|}|=|W^*|.$
	
First we assume that $X$ is connected. By Lemma \ref{prok}, if $W$ is a maximum independent set of graph  $X\odot H$, then $\dim(S/I(X\odot H))=|W|.$ We need to prove that there exists an independent set $W$ such that $|W|=|V(X)|\cdot  \big(\dim(K[V(H)]/I( H))+|i(H)|\big).$  %As for any graph $H,$ we denote the set of isolated vertices of $H$ by $i(H)$  and if $A:=V(H)\backslash i(H)$ then we denote the induced subgraph of $H$ on  $A$ by $H'.$ Clealy $\dim(K[V(H)])/I(H)=\dim(K[V(H')])/I(H').$  By Definition \ref{defcorona}, arbitrary $j^{th}$ vertex of $X$ is pairwise adjacent to each vertex in $V(H_j)$  and any two vertices in $i(H_j)$ are pairwise non-adjacent in graph $X\odot H,$ where $j\in \{1,\dots, |V(X)| \}.$ Therefore, $i(H_j)$ have a significance importance in determining the maximum independent set of $X\odot H.$ 
Let $B$ be an induced  subgraph of $X\odot H$ on vertex set $\underset{j=1}{\overset{|V(X)|}{\cup}} V(H_j),$ clearly $B$ is a disjoint union of  $|V(X)|$ graphs  where each graph is isomorphic to $H.$ If $W_j$ is a maximum independent set of the $j^{th}$ copy of $H$ in $B,$ then clearly $ W=W_1\cup W_2\dots\cup W_{|V(X)|}$ is a maximum independent set of $B.$
It is easy to see that $|W_j|=\dim(K[V(H)]/I( H))+|i(H)|.$
%\begin{equation}\label{meexi}  W=W_1\cup W_2\dots\cup W_{|V(X)|}. \end{equation} 
%Let $W_1,\dots, W_{|V(X)|}$  be the   maximum independent sets of induced subgraphs of $j$ $X\odot H$ on  $\underset{j=1}{\overset{|V(X)|}{\cup}} V(H_j),$ corresponding to	$H_1,\dots, H_{|V(X)|}$  copies of graph $H$. Here $H_j\cong H$ and  $|W_j|=\dim(K[V(H)]/I( H))+|i(H)|.$  Let us define \begin{equation}\label{meexi}  W=W_1\cup W_2\dots\cup W_{|V(X)|}.	\end{equation} Clearly, $W$ is an independent set of $X\odot H$ as  it is a union of all maximum independent sets of induced subgraph of $X\odot H$ on $\underset{j=1}{\overset{|V(X)|}{\cup}} V(H_j)$.
	We claim that $W$ is a maximum independent set of  $X\odot H.$ On contrary, suppose that there exists another maximum independent set $W'$ of $X\odot H$ such that $|W'|>|W|.$ This implies that there exists a vertex $z\in W'$ such that  $z\notin W.$ Then we have two cases to discuss that is either $z\in V(X)$ or $z\in \underset{j=1}{\overset{|V(X)|}{\cup}} V(H_j).$
	If $z\in \underset{j=1}{\overset{|V(X)|}{\cup}} V(H_j) $ then  for some $j,$  $z\notin W_j$ where $W_j$ is a maximum independent set of induced subgraph of $X\odot H$ on $V(H_j).$ This implies that $z$ is adjacent to some vertex in $W_j.$ This is a contradiction to the assumption that $W'$ is an independent set. Now assume that $z\in V(X).$ This means that for some $j,$ we have $z=y_j.$ By Definition \ref{defcorona}, $y_j$ is pairwise adjacent to each vertex in $V(H_j),$ therefore  $W'\subset V(X\odot H)\backslash N_{X\odot H}(y_j).$  Since $|N_{X\odot H}(y_j)|\geq 1,$ then $|V(X\odot H)\backslash \{y_j\}|\geq |V(X\odot H)\backslash N_{X\odot H}(y_j)|.$ This implies that the cardinality of the maximum independent set of induced subgraph of $X\odot H$ on $V(X\odot H)\backslash \{y_j\}$ is strictly greater than or equals to the cardinality of a maximum independent set of induced subgraph of $X\odot H$ on $V(X\odot H)\backslash N_{X\odot H}(y_j).$ Clearly, $W \subset V(X\odot H)\backslash \{y_j\}$ and $W'\subset V(X\odot H)\backslash N_{X\odot H}(y_j).$
	We get $|W|\geq |W'|,$ which contradicts our assumption. If $X$ is  a not a connected graph, that is, $X$ is a union of disjoint connected components say $X_1,\dots,X_l,$ then by {\cite[Lemma 1.11(3)]{hibikrull}}, we have $\dim(S/I(X\odot H))= \sum_{k=1}^{l} \dim(K[V(X_k\odot H)]/I(X_k\odot H))$   and our result follows.	\end{proof}
	
		\begin{Corollary}
	       Let $m,n,u,v,r \in \mathbb{Z}^+$ and $l,q \geq 3.$  Then
	       	\begin{itemize}
	       	    \item[(a)] 	 $\dim(K[V(P_{n}\odot P_{m})]/I(P_{n}\odot P_{m}))=n\cdot 	\big\lceil \frac{m}{2}\big\rceil.$
	       	     \item[(b)] 	 $\dim(K[V(C_{l}\odot C_{q})]/I(C_{l}\odot C_{q}))=l\cdot 	\big\lceil \frac{q-1}{2}\big\rceil.$
	       	      \item[(c)] 	 $\dim(K[V(S_{n}\odot S_{m})]/I(S_{n}\odot S_{m}))=(n+1)\cdot m.$
	       	       \item[(d)] 	 $\dim(K[V(K_{n}\odot K_{m})]/I(K_{n}\odot K_{m}))=n.$
	       	        \item[(e)] 	 $\dim(K[V(K_{m,n}\odot K_{u,v}))]/I(K_{m,n}\odot K_{u,v}))=(m+n)\cdot \max\{u,v\}.$
	       	         \item[(f)]  If $G$ is a disjoint union of $C_q$ and $N_r,$	then $\dim(K[V(P_n \odot G))]/I(P_n \odot G))=n\cdot (	\big\lceil \frac{q-1}{2}\big\rceil+r).$
	       	\end{itemize}
	\end{Corollary}

		\begin{Lemma}\label{lemcomplete}
		Let $H$ be a non-trivial connected graph and $ S=K[V(H)].$ Then $\dim(S/I(H))=1$  iff $H$ is complete graph.
	\end{Lemma}
	\begin{proof}
		Let $W$ be the maximum independent set of $H$. If $\dim(S/I(H))=1$ (i.e $|W|=1$)   implies that all the vertices of $H$ are pairwise adjacent which proves that $H$ is a complete graph. The converse statement easily followed by Lemma \ref{prok}.
	\end{proof}
	\noindent 	  % Villarreal in {\cite[Proposition 2.2]{vilareal}} proved that whiskered graph is Cohen Macaulay. If $H$ is a trivial graph, then $X\odot H$ is a whiskered graph and our next result holds.
	If $H$ is a trivial graph, then $X\odot H$ is a whisker graph of $X.$ Villarreal proved in {\cite[Proposition 2.2]{vilareal}} that whisker graph of any graph is  Cohen-Macaulay graph. So in our next theorem, we characterize all  Cohen-Macaulay graphs $X\odot H,$ if $X \in \{P_{n},C_{n},K_{n},S_{n+1},K_{u,v}\}.$
	\begin{Theorem}\label{cohnn}
		Let $H$ be any graph, $X\in \{P_{n},C_{n},K_{n},S_{n+1},K_{u,v}\}$ and $ S=K[V(X \odot H)].$ 
		%	 one of the following graphs:
		%	\begin{itemize}
		%	\item[(i)] H is $P_{n}$ with  $n\geq 2$,
		%	\item[(ii)] H is $C_{n}$ with  $n\geq 3$,
		%	\item[(iii)] H is $\overline{K}_{m}$ be a null graph,
		%		\item[(iv)] H is $K_{m}$ be a complete graph,
		%		\item[(v)] H is $D_{m+1}$ be a star graph,
		%		\item[(vi)] H is $\mathcal{K}_{u,v}$ be the complete bipartite graph.\end{itemize}
		Then $S/I(X \odot H)$ is  Cohen-Macaulay iff $H$ is a complete graph.
	\end{Theorem}
	%\begin{Theorem}
	%	Let $P_{n}$ be a path on $n\geq 2$ vertices and $H$ be any graph such that $t=\depth(S_q/I(H))=\depth(S_b/I(H^{\star}))+a=t_{1}+a\geq1$, where $H^{\star}\subseteq H$ on non-isolated vertices and $a$ are isolated vertices in $H$, then $S/I(P_{n}\odot H)$ is Cohen-Macaulay if and only if $\dim(S/I(H))=1.$
	%\end{Theorem}
	\begin{proof}  %If $H$ is a trivial graph, then $X\odot H$ is a Cohen-Macaulay graph as proved by Villarreal in {\cite[Proposition 2.2]{vilareal}} that whiskered graph is Cohen Macaulay. 
By  Lemma \ref{lemcomplete}, a non-trivial connected graph $H$ is complete iff $\dim(K[V(H)]/I(H))=\depth(K[V(H)]/I(H))=1.$ 	   We discuss all the cases one by one as follows:
	
		\begin{description}
	
	    \item[1] 	Let  $X=P_{n}.$ By  Theorem \ref{The1} and  Theorem \ref{krulldim}, the module  $S/I(P_{n}\odot H)$ is Cohen-Macaulay iff $\depth(S/I(P_{n}\odot H))=\dim(S/I(P_{n}\odot H))$ iff     $$\big\lceil\frac{n}{2}\big\rceil+\big\lceil\frac{n-1}{2}\big\rceil \big(\depth(K[V(H)]/I(H))+|i(H)|\big)=n\cdot \big(\dim(K[V(H)]/I( H))+|i(H)|\big)$$ iff $|i(H)|=1$ and $\depth(K[V(H)]/I(H))=0=\dim(K[V(H)]/I(H))$ or   $|i(H)|=0$ and $\dim(K[V(H)]/I(H))=\depth(K[V(H)]/I(H))=1$ iff  $H$ is a complete graph.\\
			\item[2] If $X=C_n,$  then by using Theorem \ref{The3} and Theorem \ref{krulldim}, the  module  $S/I(C_{n} \odot H)$ is Cohen-Macaulay iff $\depth(S/I(C_{n} \odot H))=\dim(S/I(C_{n} \odot H))$ iff  $$\big\lceil\frac{n-1}{2}\big\rceil+\big\lceil\frac{n}{2}\big\rceil \big(\depth(K[V(H)]/I(H))+|i(H)|\big)=n \cdot \big(\dim(K[V(H)]/I( H))+|i(H)|\big)$$ iff $|i(H)|=1$ and $\depth(K[V(H)]/I(H))=\dim(K[V(H)]/I(H))=0$ or   $|i(H)|=0$ and  $\dim(K[V(H)]/I(H))=\depth(K[V(H)]/I(H))=1$ iff  $H$ is a complete graph.\\
	\item[3] 
Let $X=K_{n}.$  By Theorem \ref{The119} and Theorem \ref{krulldim}, the module  $S/I(K_{n} \odot H)$ is Cohen-Macaulay iff $\depth(S/I(K_{n} \odot H))=\dim(S/I(K_{n} \odot H))$ iff 
$$1+(n-1)\big(\depth(K[V(H)]/I(H))+|i(H)|\big)=n \cdot \big(\dim(K[V(H)]/I( H))+|i(H)|\big)$$ iff $|i(H)|=1$ and $\depth(K[V(H)]/I(H))=\dim(K[V(H)]/I(H))=0$ or   $|i(H)|=0$ and  $\dim(K[V(H)]/I(H))=\depth(K[V(H)]/I(H))=1$ iff $H$ is a complete graph.  \\
\item[4] 
If $X=S_{n},$ then  by  Theorem \ref{The1199} and  Theorem \ref{krulldim},  we have
%$\depth (S/I(S_{n} \odot H))=n+\depth(K[V(H)]/I(H))+|i(H)|$ and by Theorem \ref{krulldim}, we get $\dim(S/I(S_{n} \odot H))=(n+1) \cdot \big(\dim(K[V(H)]/I( H))+|i(H)|\big).$ The module
$S/I(S_{n} \odot H)$ is Cohen-Macaulay iff $\depth (S/I(S_{n} \odot H))=\dim(S/I(S_{n} \odot H))$ iff  $$ n+\depth(K[V(H)]/I(H))+|i(H)|=(n+1) \cdot \big(\dim(K[V(H)]/I( H))+|i(H)|\big)$$  iff $|i(H)|=1$ and $\depth(K[V(H)]/I(H))=\dim(K[V(H)]/I(H))=0$ or   $|i(H)|=0$ and  $\dim(K[V(H)]/I(H))=\depth(K[V(H)]/I(H))=1$ iff $H$ is a complete graph.\\

\item[5] Let $X=K_{u,v}.$ Without loss of generailty, we assume that $v\geq u.$ Using Theorem \ref{The1177} and Theorem \ref{krulldim}, the module $S/I(K_{u,v}\odot H)$ is Cohen-Macaulay iff $\depth(S/I(K_{u,v}\odot H))=\dim(S/I(K_{u,v}\odot H))$ iff   $$u\cdot \big(\depth(K[V(H)]/I(H))+|i(H)|\big)+v=(u+v)\cdot \big(\dim(K[V(H)]/I( H))+|i(H)|\big)$$ iff $|i(H)|=1$ and $\depth(K[V(H)]/I(H))=\dim(K[V(H)]/I(H))=0$ or   $|i(H)|=0$ and  $\dim(K[V(H)]/I(H))=1=\depth(K[V(H)]/I(H))$ iff $H$ is a complete graph.
\end{description}
	\end{proof}
	Let  $G_1 \sqcup G_2$ denotes a disjoint union of two graphs $G_1$ and $G_2.$ By using Theorem {\cite[Theorem 3.2]{haghi}}, we have the following corollary.
\begin{Corollary}
           Let $H$ be any graph and $G$ be a graph such that $G=G_1\sqcup G_2\sqcup\dots\sqcup G_k,$ where $G_i \in \{P_{n},C_{n},K_{n},S_{n+1},K_{u,v}\}.$  Then $K[V(G\odot H)]/I(G\odot H)$ is Cohen-Macaulay iff $H$ is a complete graph.
\end{Corollary}
\section*{Conclusion}
		In last decades, edge ideals have garnered significant attention; see for instance \cite{vann}. Various findings on these edge ideals have demonstrated how combinatorial and algebraic aspects interact. The primary goal of this paper is to compute algebraic invariants such as depth, Stanley depth, regularity, projective dimension and Krull dimension of edge ideals associated with the corona product of two graphs and as a consequence of our findings, we characterize some Cohen–Macaulay graphs. By restricting one of the class to a well-known graph in corona product of two graphs, we are able to handle the algebraic invariants for these structures, and our results gives strong motivation for further studies into this intriguing area to compute the algebraic invariants when both classes in corona product of two graphs are arbitrary.

	\end{document}